\documentclass[twocolumn,twoside]{IEEEtran}
\usepackage{amsmath,amssymb,amsthm,epsfig,color,subfigure,empheq,graphicx,pgfplots}
\usepackage{enumerate,url,algorithm,wasysym,epstopdf,balance,enumerate}

\newcommand \ba{\mathbf{a}}
\newcommand \bb{\mathbf{b}}
\newcommand \bc{\mathbf{c}}
\newcommand \be{\mathbf{e}}
\newcommand \bg{\mathbf{g}}

\newcommand \bm{\mathbf{m}}

\newcommand \br{\mathbf{r}}

\newcommand \bv{\mathbf{v}}

\newcommand \bx{\mathbf{x}}
\newcommand \by{\mathbf{y}}
\newcommand \bz{\mathbf{z}}
\newcommand \bp{\mathbf{p}}
\newcommand \bq{\mathbf{q}}
\newcommand \bA{\mathbf{A}}

\newcommand \bC{\mathbf{C}}
\newcommand \bG{\mathbf{G}}
\newcommand \1{\mathbf{1}}
\newcommand \bPhi{\boldsymbol{\Phi}}
\newcommand \bSigma{\mathbf{\Sigma}}
\newcommand \bZero{\mathbf{0}}
\newcommand \bPi{\mathbf{\Pi}}

\newcommand \bDelta{\boldsymbol{\Delta}}
\newcommand \bdelta{\boldsymbol{\delta}}

\newcommand \bI{\mathbf{I}}
\newcommand \bE{\mathbf{E}}

\newcommand \bM{\mathbf{M}}

\newcommand \bF{\mathbf{F}}
\newcommand \bX{\mathbf{X}}

\newcommand \bR{\mathbf{R}}
\newcommand \bW{\mathbf{W}}

\newcommand \bTheta{{\mathbf{\Theta}}}
\newcommand \tbb{\tilde{\mathbf{b}}}
\newcommand \tbv{\tilde{\mathbf{v}}}

\newcommand \tbp{\tilde{\mathbf{p}}}
\newcommand \tbq{\tilde{\mathbf{q}}}
\newcommand \cbb{\check{\mathbf{b}}}

\newcommand \tbA{\tilde{\mathbf{A}}}
\newcommand \tbV{\tilde{\mathbf{V}}}

\newcommand \tbTheta{\tilde{\mathbf{\Theta}}}
\newcommand \tTheta{\tilde{\Theta}}
\newcommand \tbGamma{{\tilde{\mathbf{\Gamma}}}}
\newcommand \tGamma{\tilde{\Gamma}}

\newcommand \hbTheta{\hat{\mathbf{\Theta}}}
\newcommand \mcA{\mathcal{A}}

\newcommand \mcC{\mathcal{C}}
\newcommand \mcD{\mathcal{D}}
\newcommand \mcG{\mathcal{G}}
\newcommand \mcL{\mathcal{L}}

\newcommand \mcF{\mathcal{F}}
\newcommand \mcN{\mathcal{N}}
\newcommand \mcM{\mathcal{M}}
\newcommand \mcE{\mathcal{E}}

\newcommand \mcP{\mathcal{P}}
\newcommand \mcS{\mathcal{S}}
\newcommand \mcT{\mathcal{T}}
\newcommand \mcW{\mathcal{W}}
\newcommand \mcX{\mathcal{X}}
\newcommand \mcZ{\mathcal{Z}}

\newtheorem{assumption}{Assumption}
\newtheorem{example}{Example}

\DeclareMathOperator{\diag}{dg}

\DeclareMathOperator{\trace}{Tr}
\DeclareMathOperator{\vectorize}{vec}

\newtheorem{proposition}{Proposition}
\newtheorem{lemma}{Lemma}
\newtheorem{corollary}{Corollary}
\newtheorem{theorem}{Theorem}

\begin{document}
\title{Inverter Probing for Power Distribution\\
Network Topology Processing}
	
\author{Guido Cavraro and Vassilis Kekatos,~\IEEEmembership{Senior Member,~IEEE}}




%

\maketitle

\begin{abstract}
Knowing the connectivity and line parameters of the underlying electric distribution network is a prerequisite for solving any grid optimization task. Although distribution grids lack observability and comprehensive metering, inverters with advanced cyber capabilities currently interface solar panels and energy storage devices to the grid. Smart inverters have been widely used for grid control and optimization, yet the fresh idea here is to engage them towards network topology inference. Being an electric circuit, a distribution grid can be intentionally probed by instantaneously perturbing inverter injections. Collecting and processing the incurred voltage deviations across nodes can potentially unveil the grid topology even without knowing loads. Using grid probing data and under an approximate grid model, the tasks of topology recovery and line status verification are posed respectively as non-convex estimation and detection problems. Leveraging the features of the Laplacian matrix of a tree graph, probing terminal nodes is analytically shown to be sufficient for exact topology recovery if voltage data are collected at all buses. The related non-convex problems are surrogated to convex ones, which are iteratively solved via closed-form updates based on the alternating direction method of multipliers and projected gradient descent. Numerical tests on benchmark feeders demonstrate that grid probing can yield line status error probabilities of $10^{-3}$ by probing 40\% of the nodes. 
\end{abstract}


\section{Introduction}\label{sec:intro}
\allowdisplaybreaks
To improve reliability, provide ancillary services, and avoid network violations, utilities need to control the power injections from solar resources, energy storage devices, and microgenerators. To perform any meaningful grid monitoring and optimization task though, a detailed model for the underlying grid topology is needed. Some utilities have limited information on their primary and/or secondary networks. Others may know their line infrastructure and impedances, but not which lines are currently energized --- recall that grids are oftentimes reconfigured for maintenance, to balance loads, or alleviate faults. Topology processing could be broadly classified into \emph{topology verification} and \emph{topology identification}. In topology verification, the operator knows the line infrastructure and impedances and would like to find the energized lines. Topology identification aims to find both the connectivity and line impedances, and is hence a more challenging task.

In transmission networks, topology processing is usually handled through the generalized state estimator (GSE)~\cite[Sec.~4.10]{ExpConCanBook}. Since the GSE does not carry over to distribution systems due to lack of observability, recent research efforts exploit second-order statistics of electric quantities to recover distribution grid topologies. Reference~\cite{BoSch13} collects voltage magnitudes from all nodes, and leverages the structure of their covariance matrix and its inverse to identify the grid topology. The previous scheme has been improved by waiving the assumption on identical resistance-to-reactance ratios across lines~\cite{Deka1}, and by further utilizing prior information on power injection covariances at terminal buses~\cite{Deka4,ParkDeka}. Topology identification has also been tackled using graphical models by exploiting the mutual information of voltage data~\cite{WengLiaoRajagopal17}, or by inspecting the entries of the voltage covariance matrix~\cite{7058419,7541005}. A Wiener filtering approach using wide-sense stationary processes on radial networks is reported in~\cite{TalDeMate2017}. The aforementioned schemes rely on ensemble covariances and thus, reliable topology estimates can be obtained only after long observation intervals (e.g., hours). Given that a feeder typically undergoes 5-10 switching events on a daily basis~\cite{CavArg2014}, grid topologies may have changed during data collection.

To avoid delays, several works infer use a single or a few data snapshots. Commencing with topology verification, the problem has been cast as a spanning tree identification task given load data at all nodes and power flow readings at selected lines~\cite{sevlian2015distribution}: Exploiting the fundamental cycles of the grid graph, flow meters are optimally placed under noiseless and noisy data setups. Given a finite number of voltage snapshots, topology verification is posed as maximum likelihood and maximum \emph{a-posteriori} probability detection problems in~\cite{CaKe2017}. In \cite{zhao2017learning}, deep neural network-based classifiers are able to detect transmission line statuses; nevertheless, the standard PQ/PV power flow dataset used as input to the classifiers may not be available in distribution grids. Reference~\cite{patopa} recognizes that line impedances appear linearly in the power flow equations, and estimates them via a total least squares fit. With injections and voltage phasors metered at all nodes, the same work tackles topology identification using a binary search on the related likelihoods. In \cite{Ardakanian17}, given bus voltage and current phasors, the admittance matrix of a possibly meshed grid is found via linear regression. If data are missing only from zero-injection buses, the Kron-reduced admittance matrix can be recovered via a low rank-plus-sparse decomposition. 

Different from the previous schemes, topology processing can be accomplished by \emph{actively} rather than \emph{passively} collecting grid data. A microgrid topology can be inferred by transmitting power line communication signals and then measuring their reception time at different buses~\cite{ErTpVi13}. Reference \cite{Scaglione2017} perturbs the primary droop parameters of micro-generators and acquires a least-square estimate of the bus admittance matrix through the grid response. Leveraging the communication, actuation, and sensing capabilities of smart inverters, our previous works \cite{BheKeVe2017} and \cite{BKVZ17} purposefully probe the grid by varying the (re)active power injections at selected buses, record the incurred voltage responses, and thus, infer the complex loads at non-actuated buses. The transition matrix of a linear dynamical system is deciphered using active perturbations in \cite{6203379}, \cite{6966724}; here, grid dynamics are ignored due to timescales.

Building on the idea of active data collection, this work engages smart inverters to actuate the grid and learn network topologies upon collecting voltage deviations across nodes. Our contribution is on three fronts: First, we extend the idea of grid probing to the pertinent tasks of network topology identification and verification. Using non-synchonized data and under an approximate grid model, the two tasks are posed as non-convex data fits to estimate or detect the Laplacian matrix of a tree graph. Second, leveraging radial structures, we provide sufficient conditions on the placement of probing buses for successful topology recovery. It is shown that probing the candidate leaf nodes is sufficient. Third, we put forth convex relaxations of the original problems and devise efficient algorithms for solving them. 

Compared to existing works on distribution grid topology identification, our key differences are on the data acquisition and the analytical fronts. Regarding data acquisition, previous works passively collect grid data from which they estimate second-order moments of voltages and/or injections~\cite{BoSch13}, \cite{Deka4}, \cite{TalDeMate2017}, \cite{Ardakanian17}, \cite{CaKe2017}, \cite{ParkDeka}. Our inverter probing scheme accelerates data acquisition and simplifies data modeling at the expense of perturbing the grid. 

On the analytical side, we show that if: \emph{a1)} a grid is perturbed at all terminal buses, and \emph{a2)} voltage responses are collected at all buses, then its topology is identifiable by solving a non-convex problem. This non-convex problem is subsequently relaxed to a convex surrogate, whose performance is only numerically evaluated. Assumption \emph{a1)} can be waived when using the convex relaxation-based solver, though again there are no identifiability guarantees. Assumption \emph{a2)} has been previously adopted in~\cite{BoSch13}, \cite{TalDeMate2017}, and \cite{CaKe2017}; and is required for our convex solver of Section~\ref{subsec:relaxid}. The setup where data are collected only on a subset of buses has been considered in \cite{Deka4}, \cite{ParkDeka}, \cite{Ardakanian17}. The latter works establish that the grid topology is identifiable only under certain conditions, e.g., that all non-metered buses are zero-injection buses, or that they are connected to at least three metered nodes. In a nutshell, other works have waived \emph{a2)} to recover a reduced grid graph, whereas this work requires \emph{a2)} as the only practical means of recovering the actual topology.
\color{black}

The rest of this work is organized as follows. Section~\ref{sec:model} revisits graph theory preliminaries and an approximate grid model. Section~\ref{sec:identification} engages grid probing for topology identification, provides identifiability conditions, and solves the relaxed problem iteratively using closed-form updates based on the alternating direction method of multipliers (ADMM). Section~\ref{sec:verification} poses topology verification as a binary detection problem under tree Laplacian matrix constraints, establishes verifiability, and solves the relaxed minimization via a projected gradient descent scheme. Our claims are corroborated using the IEEE 13- and 37-bus feeders in Section~\ref{sec:tests}. Conclusions and future directions are summarized in Section~\ref{sec:conclusions}.

Regarding \emph{notation}, lower- (upper-) case boldface letters denote column vectors (matrices). Calligraphic symbols are reserved for sets. Symbol $^{\top}$ stands for transposition. Vectors $\mathbf{0}$ and $\mathbf{1}$ are the all-zero and all-one vectors, while $\be_m$ is the $m$-th canonical vector. Symbol $\|\mathbf{x}\|_2$ denotes the $\ell_2$-norm of $\mathbf{x}$ and $\diag(\mathbf{x})$ defines a diagonal matrix having $\mathbf{x}$ on its diagonal. A symmetric positive (semi)definite matrix is denoted as $\mathbf{X}\succ \mathbf{0}$ $(\mathbf{X}\succeq \mathbf{0})$, while $|\bX|$ and $\trace(\bX)$ are the determinant and trace of $\bX$. The sets of symmetric and psd $N\times N$ matrices are denoted by $\mathbb S^N$ and $\mathbb S^N_+$, respectively.

\section{Modeling Preliminaries}\label{sec:model}
Before presenting our probing scheme, this section reviews concepts from graph theory and an approximate grid model. 

\subsection{Preliminaries from graph theory}\label{subsec:graph}
Let $\mathcal{G}=(\mathcal{V},\mathcal{E})$ be an undirected graph, where $\mathcal{V}$ is the set of nodes and $\mathcal{E}$ the set of edges $\mathcal{E}:=\{(m,n): m,n \in \mathcal{V}\}$. A rooted tree is a connected graph without loops with one node designated as the root and indexed by 0. In a tree graph, a \emph{path} is the unique sequence of edges connecting two nodes. The set of nodes adjacent to the edges forming the path between nodes $n$ and $m$ will be denoted by $\mathcal{P}_{n,m}$. The nodes belonging to $\mcA_m:=\mcP_{0,m}$ are termed the \emph{ancestors} of node $m$; see also Fig.~\ref{fig:graph_bckgrnd}. If $n\in\mcA_m$, then $m$ is a \emph{descendant} of node $n$. The descendants of node $m$ comprise the set $\mcD_m$. By convention, $m \in \mcA_m$ and $m\in \mcD_m$. If $n \in \mcA_m$ and $(m,n)\in \mathcal E$, node $n$ is the \emph{parent} of $m$. A node without descendants is called a \emph{leaf}. Leaf nodes are collected in the set $\mcF$, while non-leaf nodes will be termed \emph{internal} nodes. 

The \emph{depth} $d_m$ of node $m$ is defined as the number of its ancestors, i.e., $d_m:= |\mcA_m|$. If $n\in\mcA_m$ and $d_n=k$, node $n$ is the unique $k$-depth ancestor of node $m$ and will be denoted by $\alpha_{m}^k$ for $k=0,\ldots,d_m$. Finally, the \emph{$k$-th level set} of node $m$ is defined as
\begin{equation}\label{eq:levelset}
\mcN_m^k := \left\{\begin{array}{ll}
\mcD_{\alpha^k_m} \setminus \mcD_{\alpha^{k+1}_m}&,~ k=0,\ldots,d_m-1\\
\mcD_m&,~ k=d_m.
\end{array}\right.
\end{equation}
The concept of the level set is illustrated in Figure~\ref{fig:graph_bckgrnd}. In essence, the level set $\mcN_m^k$ consists of node $\alpha^k_m$ and all the subtrees rooted at $\alpha^k_m$ excluding the one containing node $m$. Level sets feature the ensuing properties that will be used later.

\begin{lemma}\label{le:Nmk}
Let $m$ be a node in a tree graph.
\renewcommand{\labelenumi}{\roman{enumi})}
\begin{enumerate}
\item Node $\alpha^{k}_m$ with $0\leq k \leq d_m$, is the only node in $\mcN_m^k$ at depth $k$, while the remaining nodes in $\mcN_m^k$ are at larger depths; 
\item For all $n\in\mcN_m^k$, then $\alpha^{k}_n=\alpha^{k}_m$ and $\alpha^{k}_n\in\mcN_m^k$; and 
\item If $m$ is a leaf node, then $\mcN_m^{d_m}=\{m\}$. 
\end{enumerate}
\end{lemma}
The properties of Lemma \ref{le:Nmk} follow readily after observing that given a node $m$ and a depth $k\in\{0,\ldots,d_m\}$, the level set $\mcN_m^k$ is a subset of $\mcD_{\alpha^k_m}$. Representative examples of these properties are illustrated in the grid of Figure~\ref{fig:graph_bckgrnd}.

\begin{figure}[t]
\centering
\includegraphics[width=0.42\textwidth]{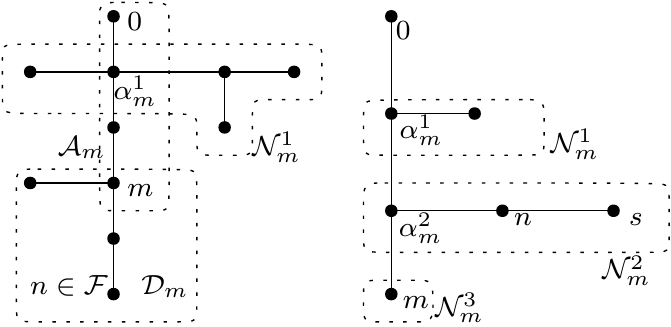}
\caption{Node $n$ is a leaf node, while $m$ is an internal node of the left graph. The ancestor ($\mcA_m$) and descendant ($\mathcal{D}_m$) sets for node $m$ are also shown. The level set $\mcN_m^1$ consists of $\alpha_m^1$ and the subtrees rooted at $\alpha_m^1$ excluding the subtree containing node $m$. The properties of Lemma~\ref{le:Nmk} can be easily checked in the right graph.}
\label{fig:graph_bckgrnd}
\end{figure}

\subsection{Power flow model}
A radial single-phase power distribution grid having $N+1$ buses can be modeled by a graph $\mcG_o=(\mcN,\mcL)$. The nodes in $\mcN:=\{0,\ldots,N\}$ represent grid buses, and the edges in $\mcL$ distribution lines. The active (reactive) power injection at bus $n$ is denoted by $p_n$ ($q_n$), while $v_n$ is its voltage magnitude. The substation bus is indexed by $n=0$ and its voltage is fixed at $v_0=1$. Power distribution networks are oftentimes operated in a radial structure, so that $\mcG_o$ is a tree rooted at the substation. The voltage magnitudes and power injections at all buses excluding the substation are collected accordingly in vectors $\bv$, $\bp$, and $\bq$. 

Let $r_\ell+j x_\ell$ be the impedance of line $\ell$, and collect all the impedances in vector $\br+j\bx$. The grid connectivity is captured by the branch-bus incidence matrix $\tbA \in\{0,\pm1\}^{L\times (N+1)}$ that can be partitioned into its first and the rest of its columns as $\tbA=[\mathbf{a}_0~\bA]$. For a radial grid, the \emph{reduced incidence matrix}~$\bA$ is square and invertible~\cite{GodsilRoyle}. Since $\tbA \mathbf 1 = \mathbf{0}$, it follows that
\begin{equation}\label{eq:a0}
\mathbf{a}_0 = -\bA \mathbf{1}.
\end{equation}

Although power injections are non-linearly related to nodal voltage phasors, upon linearizing complex power injections around the flat voltage profile $\mathbf{1}+j\mathbf{0}$, the bus voltage magnitudes can be approximated by~\cite{Saverio2}
\begin{equation}\label{eq:model}
\bv = \bR_o\bp + \bX_o\bq + \1
\end{equation}
where $\bR_o^{-1}:=\bA^\top \diag^{-1}(\br)\bA$ and $\bX_o^{-1}:=\bA^\top \diag^{-1}(\bx)\bA$. A key property of $\bR_o$ is that its $(m,n)$-th entry equals the sum of the line resistances between the nodes in $\mcA_m \cap \mcA_n$~\cite{Deka4}
\begin{equation}\label{eq:entriesR1}
[R_o]_{mn} = \sum_{\substack{\ell = (c,d) \in \mcL \\ c,d\in \mcA_m \cap \mcA_n}} r_{\ell}.
\end{equation}
Based on \eqref{eq:entriesR1}, one can recognize that the entry $[R_o]_{mn}$ equals the voltage drop from $v_0$ incurred at bus $m$ when a unitary active power is withdrawn at bus $n$ while the remaining buses are unloaded. Leveraging this interpretation, three useful properties of $\bR_o$ are presented next.

\begin{lemma}\label{le:entriesR2}
If $m$, $n$, and $s$ are nodes in a grid represented by a tree graph, then
\renewcommand{\labelenumi}{\roman{enumi})}
\begin{enumerate}
\item $[R_o]_{mm} \geq [R_o]_{mn}$ for all $n\neq m$, with strict inequality if $m$ is a leaf;
\item two nodes $n$ and $s$ belong to $\mcN_m^k$ for some $k$ if and only if $[R_o]_{mn} = [R_o]_{ms}$; and
\item if $n\in \mcN_m^k$ and $s\in \mcN_m^{k+1}$, then $[R_o]_{mn} = [R_o]_{ms} + r_{\left(\alpha^{k}_m,\alpha^{k+1}_m\right)} > [R_o]_{ms}$.
\end{enumerate}
\end{lemma}

\section{Grid Probing for Topology Identification}\label{sec:identification}
Existing topology processing schemes rely on passively collected smart meter data. Here, we put forth an active data acquisition protocol for topology identification. The idea is to leverage the communication, actuation, and sensing functionalities of smart inverters. An inverter can be commanded to shed solar generation, (dis)-charge a battery, or change its power factor within milliseconds. This enables a new data collection paradigm, where the system operator purposefully probes the electric grid by changing inverter injections and measuring the electric circuit response to possibly identify the grid topology.

Let us first model the data collected via probing. The buses hosting controllable inverters comprise set $\mcC \subseteq \mcN$ with $C:=|\mcC|$. Consider the probing action at time $t$. Each bus $m\in\mcC$ perturbs its active injection by $\delta_m(t)$. All inverter perturbations are stacked in $\bdelta(t)\in\mathbb{R}^C$. The incurred perturbation in voltage magnitudes $\tbv(t):=\bv(t)-\bv(t-1)$ is expressed from \eqref{eq:model} as
\begin{equation}\label{eq:dv}
\tbv(t) = \bR_o\bI_\mcC \bdelta(t) + \mathbf{e}(t)
\end{equation}
where the $N\times C$ matrix $\mathbf{I}_\mcC$ collects the canonical vectors associated with the buses in $\mcC$. The error vector $\mathbf{e}(t)$ captures measurement noise, the modeling error introduced by the LDF approximation, and voltage deviations attributed to possible load variations during probing.

The grid is perturbed over $T$ probing periods, each one lasting for a second or so. Given the larger resistance values of distribution lines, the modes of a distribution network are expected to be in the order of microseconds. Assuming that inverters have been commanded not to perform frequency or voltage control (at last for the period of probing), the duration of one second or so allows for a steady-state analysis of the feeder. Stacking the probing actions $\{\bdelta(t)\}_{t=1}^T$, the measured voltage deviations $\{\tbv(t)\}_{t=1}^T$, and the error terms $\{\be(t)\}_{t=1}^T$ as columns of matrices $\bDelta$, $\tbV$, and $\bE$ accordingly, yields
\begin{equation}\label{eq:dV}
\tbV = \bR_o \bI_{\mcC} \bDelta + \bE.
\end{equation}

Matrix $\bDelta\in\mathbb{R}^{C\times T}$ can be designed to be full row-rank, e.g., by setting $T=C$ and $\bDelta_1:=\diag(\{\delta_m\})$ with $\delta_m\neq 0$ for all $m\in\mcC$. A diagonal $\bDelta$ enjoys simpler synchronization, since each inverter probes the grid at different time slots. Each entry $\delta_m$ can be selected as the maximum active power deviation, inverter $m$ can implement. For example, if inverter $m$ produces solar energy $p_m^g$, it can drop it instantaneously to zero, so that $\delta_m=p_m^g$. Another meaningful choice is $\bDelta_2:= \diag(\{\delta_m\})\otimes [+1~-1]$, where $\otimes$ is the Kronecker product. In this case $T=2C$ and each inverter induces two probing actions: It first drops its generation from $p_m^g$ to zero yielding a perturbation of $\delta_m=p_m^g$. It then resumes generation at its nominal value $p_m^g$, thus incurring a perturbation of $-\delta_m$. 

The model in \eqref{eq:dV} and our subsequent developments consider perturbing active power injections yielding rise to \emph{active probing}. The scheme carries over to \emph{reactive probing}. Although reactive probing may be more practical since no actual power is curtailed, its magnitude may be confined by inverter apparent power constraints during periods of high solar generation. Beyond these implementation issues, the two probing modes are virtually identical. Active probing can be first used to infer the topology and line resistances. Then, line reactances can be recovered using a least-squares fit on data collected from reactive probing. The roles in this process can be obviously reversed. A joint active and reactive probing scheme is not recommended, since that would complicate modeling significantly.

Grid probing is definitely an invasive technique. Perturbing (re)active injections introduces instantaneous voltage variations on customer buses. Power electronics are known to introduce power quality issues with harmonics, and probing can add further to that. Nonetheless, regulation standards do tolerate voltage excursions from the desired band for short periods of time~\cite{ansic84}. Such events happen frequently due to startup currents of induction motors, switching of capacitor banks, and the natural fluctuations in power output of rooftop photovoltaics; real-world data from the Pecan Str dataset indicate that solar generation can drop by 80\% within 10 sec~\cite{pecandata}. Then, solar-enabled grids need to deal with voltage excursions anyway. Probing thus provide an additional toolbox for utilities to use, perhaps in tandem with techniques relying on passively collected grid data.
\color{black}

\subsection{Problem formulation}\label{subsec:problem-id}
Topology identification can be now posed as the task of recovering $\bR_o$ given $(\tbV,\bDelta)$ using the data model of \eqref{eq:dV}. Interpreting power perturbations and voltages respectively as the inputs and outputs of a system, topology processing can be posed as a system identification problem. This a major advantage over the existing \emph{blind} identification schemes that rely on second-order statistics of voltages (outputs)~\cite{BoSch13,Deka1}. Beyond requiring fewer data and shortening the acquisition time, the new method waives restricting statistical assumptions, such as that power injections are uncorrelated across nodes \cite{Deka1,Deka4}; or independent across time \cite{BoSch13,WengLiaoRajagopal17}. Moreover, different from other system identification-based approaches \cite{patopa,Ardakanian17}, only perturbations on power injections at probed buses are required. 

Rather than casting topology identification as a general linear system identification task, one could further exploit the properties of $\bR_o$: From the definition of $\bR_o$ in \eqref{eq:model}, its inverse
\begin{equation}\label{eq:reducedLapl}
\bTheta_o:=\bR_o^{-1}=\bA^\top \diag^{-1}(\br)\bA
\end{equation}
is a \emph{reduced weighted Laplacian} for the underlying graph $\mcG_o$ with weights equal to the inverse line resistances. 

The grid topology is equivalently captured by the non-reduced Laplacian matrix $\tbTheta_o=\tbA^\top \diag^{-1}(\br)\tbA$. To see this, note that the two matrices are related through the linear mapping $\bPhi: \mathbb S^N \rightarrow \mathbb S^{N+1}$ as [cf.~\eqref{eq:a0}]
\begin{equation}\label{eq:Phi}
\tbTheta_o=\bPhi(\bTheta_o):=\begin{bmatrix}
\1^\top \bTheta_o \1 & - \1^\top \bTheta_o\\
- \bTheta_o \1 & \bTheta_o
\end{bmatrix}.
\end{equation}
As a consequence, finding $\bTheta_o$ is equivalent to finding $\tbTheta_o$. Because the entry $[\tilde \Theta_o]_{mn}$ equals $ - r_{(m,n)}^{-1}$ if nodes $m$ and $n$ are directly connected, and zero otherwise, we aim at estimating $\bTheta_o$ to unveil the grid topology and estimate line resistances. Line reactances can be similarly found via reactive probing.

We will next see how the properties of $\tbTheta_o$ translate to $\bTheta_o$. Being the Laplacian of a connected graph, matrix $\tbTheta_o$ is known to be symmetric positive semidefinite ($\tbTheta_o \succeq \mathbf{0}$) having zero as a simple eigenvalue and $\1$ as the related eigenvector~\cite{chung1997spectral}
\begin{equation}\label{eq:Theta1=0}
\tbTheta_o \1 = \mathbf{0}.
\end{equation}
Matrix $\bTheta_o$ is symmetric strictly positive definite ($\bTheta_o\succ \mathbf{0}$), due to \eqref{eq:reducedLapl} and the fact that $\bA$ is non-singular~\cite{VKZG16}. The sign information on the off-diagonal entries of $\tbTheta_o$ carries over to $\bTheta_o$. From \eqref{eq:Phi}--\eqref{eq:Theta1=0} and because the first column of $\tbTheta_o$ has non-positive off-diagonal entries, it also follows that $\bTheta_o \1 \geq \mathbf{0}$.

The grid operator may know that two specific buses are definitely connected, e.g., through flow sensors or line status indicators. To model known line statuses, let us introduce matrix $\tbGamma\in\mathbb{S}^{N+1}$ with $[\tGamma]_{mn}=0$ if line $(m,n)$ is known to be non-energized; and $[\tGamma]_{mn}=1$ if there is no prior information for line $(m,n)$. If there is no information for any line, then apparently $\tbGamma = \mathbf 1 \mathbf 1^\top$. Based on $\tbGamma$, define the set
\begin{equation*}
\mcS(\tbGamma) := \left\{
\bTheta:
\begin{array}{l}
\Theta_{mn}\leq 0,~\text{if}~[\tGamma]_{mn} = 1 \\
\Theta_{mn} = 0,~\text{if}~[\tGamma]_{mn} = 0
\end{array} m,n\in\mcN, m\neq n
\right\}.
\end{equation*}
The set $\mcS(\tbGamma)$ ignores possible prior information on lines fed directly by the substation. This information is encoded on the zero-th column of $\tbGamma$. In particular, if $[\tGamma]_{0n}=1$, then $[\tTheta]_{0n}\leq 0$ and $\sum_{m=1}^N[\Theta]_{mn}\geq 0$. Otherwise, it holds that $[\tTheta]_{0n}=\sum_{m=1}^N[\Theta]_{mn}=0$. The two properties related to lines directly connected to the substation are captured by the set
\begin{equation*}
\mcS_0(\tbGamma) := \left\{
\bTheta:
\begin{array}{l}
\mathbf{e}_n^\top\bTheta \1 \geq 0,~\text{if}~[\tGamma]_{0n} = 1 \\
\mathbf{e}_n^\top\bTheta \1 = 0,~\text{if}~[\tGamma]_{0n} = 0
\end{array}, n\in \mcN
\right\}.
\end{equation*}

Summarizing, the set of admissible reduced Laplacian matrices for arbitrary graphs with prior edge information $\tbGamma$ is
\begin{equation}\label{eq:mcM}
\mcM:=\left\{\bTheta: \bTheta \in \mcS(\tbGamma)\cap \mcS_0(\tbGamma), \bTheta = \bTheta^\top \right\}.
\end{equation}
The set $\mcM$ is convex since it is described by a system of linear (in)equalities in $\bTheta$ for given $\tbGamma$. By invoking the Gershgorin's disc theorem, it can be shown that $\bTheta\succeq \mathbf{0}$ for all $\bTheta\in\mcM$, that is $\mcM\subseteq \mathbb{S}^+$. The reduced Laplacian matrices included in $\mcM$ correspond to possibly meshed and/or disconnected graphs. 

Enforcing two additional properties on $\bTheta$ can render it a proper reduced Laplacian of a tree network. First, the Laplacian matrix in its non-reduced form $\bPhi(\bTheta)$ should have exactly $2N$ non-zero off-diagonal entries related to the $N$ edges of a tree with $N+1$ nodes. Second, the reduced Laplacian $\bTheta$ should be strictly positive definite since the network is connected. These two properties are modeled as
\begin{equation}\label{eq:mcT}
\mcT:=\left\{\bTheta: \|\bPhi(\bTheta)\|_{0,\text{off}}=2N, \bTheta\succ \mathbf{0}\right\}
\end{equation}
where $\|\mathbf{X}\|_{0,\text{off}}$ counts the number of non-zero off-diagonal entries of $\mathbf{X}$. The set $\mcT$ encodes the information that tree networks have the \emph{sparsest} $\bPhi(\bTheta)$'s associated with \emph{non-singular} $\bTheta$'s. Unfortunately, the constraint $\|\bPhi(\bTheta)\|_{0,\text{off}}=2N$ is non-convex: the non-zero entry count of a vector is a non-convex function since for example $\|\theta \be_1+(1-\theta)\be_2\|_0> \theta \|\be_1\|_0 +(1-\theta)\|\be_2\|_0$ for $\theta=1/2$ and $\be_i$ being the $i$-th canonical vector. Moreover, the set $\bTheta\succ \mathbf{0}$ is open since it does not include the boundary of the positive semidefinite cone.

Let us now return to the task of topology recovery using the model in \eqref{eq:dV}. An estimate $\hat \bTheta$ of $\bTheta_o$ could be obtained via a (weighted) least-square (LS) fit of the probing data under the Laplacian constraints, that is
\begin{equation}\label{eq:id}
\min_{\bTheta\in \mcM\cap \mcT}f(\bTheta) :=  \frac{1}{2} \|\bTheta \tbV - \bI_{\mcC} \bDelta\|_{\bW}^2.
\end{equation}
The matrix norm $\|\bX\|_\mathbf{W}$ is defined for a weighting matrix $\bW \succ \mathbf{0}$ as $\|\bX\|_\mathbf{W}^2:=\trace(\bX^\top \bW\bX)=\|\bW^{1/2}\bX\|_F^2$ with $\|\bX\|_F$ being the Frobenius norm of $\bX$.

Setting $\bW=\bI_N$ is the simplest option for the weighting matrix. If additional information on the probing data is available, other options can be used. For example, if the error term in \eqref{eq:dv} is primarily attributed to load variations $\tbp_L(t):=\bp_L(t)-\bp_L(t-1)$ and $\tbq_L(t):=\bq_L(t)-\bq_L(t-1)$, where $\bp_L$ and $\bq_L$ represent the active and reactive power absorbed by loads, the error term can be approximated as
\begin{equation*}
\mathbf{e}(t)\simeq \bR_o\tbp_{L}(t)+\bX_o \tbq_{L}(t).
\end{equation*}
If we further assume that $\bX_o=\gamma \bR_o$ for a complex scalar $\gamma$ as in \cite{CaKe2017}, the error term becomes 
\begin{equation}\label{eq:e_approx}
\mathbf{e}(t)\simeq \bR_o [\tbp_L(t)+\gamma \tbq_L(t)].
\end{equation}
The latter assumption can be justified because the reactance-to-resistance ratios $x_\ell/r_\ell$ do not vary significantly across distribution lines; see e.g., \cite[Table~I]{CaKe2017}. Approximating these ratios as $x_\ell/r_\ell=\gamma$ for all $\ell\in \mathcal{E}$ provides $\bX_o\simeq \gamma \bR_o$.

Plugging \eqref{eq:e_approx} into \eqref{eq:dv} and premultiplying by $\bTheta_o=\bR_o^{-1}$ yields the data model
\begin{equation}\label{eq:dv2}
\bTheta_o \tbv(t) =\bI_\mcC \bdelta(t) + [\tbp_L(t)+\gamma \tbq_L(t)].
\end{equation}
Postulating that load variations are zero-mean with known covariance matrices, the error term in the right-hand side (RHS) of \eqref{eq:dv2} becomes zero-mean with covariance matrix
\begin{equation}\label{eq:Sigma}
\bSigma:= \bSigma_p + \gamma^2 \bSigma_q + \gamma \left(\bSigma_{pq} + \bSigma_{pq}^\top\right)
\end{equation}
where $\bSigma_p:=\mathbb{E}[\tbp_L(t)\tbp^\top_L(t)]$, $\bSigma_q:=\mathbb{E}[\tbq_L(t)\tbq^\top_L(t)]$, and $\bSigma_{pq}:=\mathbb{E}[\tbp_L(t)\tbq^\top_L(t)]$.
The minimizer of \eqref{eq:id} with $\bW=\bSigma^{-1}$ becomes the best linear unbiased estimator (BLUE) of $\bTheta_o$~\cite[Th.~6.1]{kay93book}. If load variations are Gaussian, the minimizer of \eqref{eq:id} is the maximum likelihood estimate (MLE) of $\bTheta_o$.

\subsection{Identifiability analysis}\label{subsec:ia}
This section studies whether the actual topology can be uniquely recovered by probing the buses in $\mcC$. As customary in identifiability analysis, the probing data in \eqref{eq:model} are considered noiseless $(\mathbf{E}=\mathbf{0})$. In general, the true Laplacian $\bTheta_o$ may not be identifiable, i.e., the solution of \eqref{eq:id} may not be unique. Nevertheless, the next result shows that the graph associated with the recovered Laplacian $\bTheta$ shares some properties with the actual grid graph.

\begin{proposition}\label{pro:sameLS}
Let $\mcG$ be the radial graph associated with the Laplacian matrix $\bTheta_o$. Given probing data $(\tbV,\bDelta)$ where $\tbV = \bTheta_o^{-1} \bI_{\mcC} \bDelta$ and $\textrm{rank}(\bDelta)=C$, let $\bTheta$ be a solution of \eqref{eq:id} and $\mcG'$ its associated graph. Then, for every probing bus $m \in \mcC$, it holds that
\begin{align}
&\mcN^k_m(\mcG) = \mcN^k_m(\mcG'), \quad \forall k = 0,\dots, d_m \label{eq:samelevelset}\\
&d_m(\mcG) = d_m(\mcG'). \label{eq:samedeep}
\end{align}
\end{proposition}

\begin{proof}
Matrix $\bTheta$ is a minimizer of \eqref{eq:id} if and only if $\bTheta^{-1}\bI_{\mcC}\bDelta=\bTheta^{-1}_o\bI_{\mcC}\bDelta$. Because $\textrm{rank}(\bDelta)=C$, it follows that $\bTheta^{-1}\bI_{\mcC}=\bTheta^{-1}_o\bI_{\mcC}$, or for $\bR := \bTheta^{-1}$ that
\begin{equation*}
\bR \be_m = \bR_o \be_m, \quad \forall m\in \mcC. 
\end{equation*}
Hence, Lemma~\ref{le:entriesR2}-(ii) implies that $d_m(\mcG') = d_m(\mcG)$ and 
\begin{equation*}
\mcN^{k}_m(\mcG') = \mcN^{k}_m(\mcG)
\end{equation*}
for all $k=1,\ldots,d_m$.
\end{proof}

\begin{figure}[t]
\centering
\includegraphics[width=0.30\textwidth]{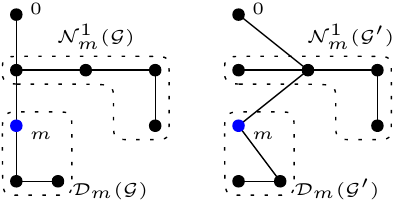}
\caption{Let $m$ be the only probing bus. The left panel depicts the actual grid topology, while the right panel depicts one of the possible minimizers of \eqref{eq:tv}. The topologies differ in the connections among buses in $\mcD_m$ or $\mcN_m^1$.}
\label{fig:LemmaLS}
\end{figure}

Proposition~\ref{pro:sameLS} ensures that probing buses are at the same depth and have the same level sets in the original graph $\mcG$ and the recovered one $\mcG'$; see Fig.~\ref{fig:LemmaLS} for an example. 

Based on the aforesaid partial results, we next study the complete identifiability of $\mcG$. Some intermediate claims proved in the appendix precede our main result. Lemma~\ref{le:levelsets} and Corollary~\ref{co:Wexample} relate ancestor nodes to level sets. An example for the former is illustrated in Fig.~\ref{fig:leaf+id}. 

\begin{figure}[t]
\centering
\includegraphics[width=0.25\textwidth]{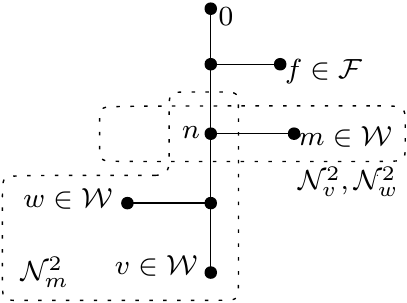}
\caption{Pictorial representation of Lemma~\ref{le:levelsets}.}
\label{fig:leaf+id}
\end{figure}

\begin{lemma}\label{le:levelsets}
In a tree graph, the node $n$ is the $k$-depth ancestor of a leaf node $m\in\mcF$, that is $n = \alpha^k_m$, if and only if there exists a subset of leaves $\mcW\subseteq \mcF$ with $m\in\mcW$ such that 
\begin{equation}\label{eq:levelsets}
\{n\}= \bigcap \limits_{w \in \mcW} \mcN^k_w.
\end{equation}
In words, node $n$ is the singleton intersection among the $k$-th level sets of all nodes in $\mcW$.
\end{lemma}

\begin{corollary}\label{co:Wexample}
In a tree graph, let $n$ be a $k$-depth node and let $\mcW := \mcD_n \cap \mcF$. Then, $\{n\}= \bigcap \limits_{w \in \mcW} \mcN^k_w$.
\end{corollary}

Building on Lemma~\ref{le:levelsets} and Corollary~\ref{co:Wexample}, the ensuing result establishes a sufficient condition for recovering the connectivity of a tree. 

\begin{lemma}\label{le:uniqueness}
In a tree $\mcG = (\mcN, \mcE)$, the edge set $\mcE$ is uniquely characterized by the set of leaf nodes $\mcF(\mcG)\subseteq \mcN$ and the level sets $\{\mcN_w^k(\mcG)\}_{k=0}^{d_w}$ for all $w\in \mcF(\mcG)$.
\end{lemma}

Using the topology recovery condition of Lemma~\ref{le:uniqueness}, our main identifiability result follows. 

\begin{theorem}\label{th:leaf+id}
Given probing data $(\tbV,\bDelta)$ where $\tbV = \bTheta_o^{-1} \bI_{\mcC} \bDelta$ and $\textrm{rank}(\bDelta)=C$, the resistive network topology is identifiable if the grid is probed at all leaf nodes, that is $\mcF\subseteq \mcC$, and voltage data are collected at all nodes.
\end{theorem}

\begin{IEEEproof}
The topology is unidentifiable if there exists a $\bTheta\in \mcM \cap \mcT$ with $\bTheta\neq \bTheta_o$ satisfying $\bTheta^{-1} \bI_{\mcC} \bDelta=\bTheta_o^{-1} \bI_{\mcC} \bDelta$. Proposition~\ref{pro:sameLS} already guarantees that for every probing leaf bus $m \in \mcC$, it holds that $\mcN(\mcG)_m^k = \mcN(\mcG_o)_m^{k}$ for $k=0,\ldots,d_m$. However, Lemma~\ref{le:uniqueness} ensures that if $\mcF \subseteq \mcC$, the connectivity of $\mcG$ equals the connectivity of $\mcG_o$.

So far, we have shown that $\mcG$ exhibits the same connectivity with $\mcG_o$. Nonetheless, the recovered resistances may not agree with the actual resistances. To waive this possibility, consider the edge $\ell = (\alpha_m^k,\alpha_m^{k+1})$ for a leaf node $m$ and for a $k\in\{0,\ldots,d_m^k\}$. For this line, it holds
\begin{align*}
r_{\ell}(\mcG) & = [R]_{\alpha_m^km} - [R]_{\alpha_m^{k+1}m}\\
& = [R_o]_{\alpha_m^km} - [R_o]_{\alpha_m^{k+1}m}\\
&= r_{\ell}(\mcG_o)
\end{align*}
where the first and third equalities originate from Lemma~\ref{le:entriesR2}; and the second one follows from $\bR\be_m=\bR_o\be_m$. We have shown that the edges in $\mcG$ and $\mcG_o$ have identical weights, and hence, $\mcG = \mcG_o$.
\end{IEEEproof}

Theorem~\ref{th:leaf+id} establishes that the grid topology is identifiable if the grid is probed at all leaf nodes and voltages are collected at all nodes. Under this setup, one needs at least $T=|\mcF|$ probing actions, which can be significantly fewer than the total number of nodes $N$. Other approaches instead guarantee grid identifiability after collecting $T \geq N$ data, e.g., see \cite{Ardakanian17}. When not all leaf nodes are probed, a part of the network can still be recovered. Assume that, after removing the descendants of buses in $\mcC$ , we obtain a radial network whose leaves are probing buses. Then, by combining Proposition~\ref{pro:sameLS} and Lemma~\ref{le:levelsets}, it can be shown that edge mistakes can only occur between the descendants of probing nodes; see the example of Fig.~\ref{fig:Rem3}.

The major role of leaf nodes in grid topology learning has already been identified in~\cite{Deka4}--\cite{ParkDeka}. Data here are collected actively via probing, whereas in \cite{Deka4}--\cite{ParkDeka} by passively collecting smart meter readings on a subset of buses $\mcC$. Apart from data acquisition, the two approaches differ also in the information used. In Theorem~\ref{th:leaf+id}, the information extracted from probing is essentially $\tbV \bDelta^+ = \bR_o \bI_{\mcF}$, where $\bDelta^+$ is the pseudo-inverse of $\bDelta$. The work in \cite{Deka4}--\cite{ParkDeka} on the other hand, operates on $\bI_{\mcF}^\top\bR_o\bI_{\mcF}$ presuming the ensemble voltage-injection covariances are known. Therefore, Theorem~\ref{th:leaf+id} quantifies the advantage of knowing $\bR_o\bI_{\mcF}$ rather than $\bI_{\mcF}^\top\bR_o\bI_{\mcF}$.

Note that the fundamental limit of $T=|\mcF|$ probing actions established by Theorem~\ref{th:leaf+id} presumes noiseless data. In the presence of noise, increasing the number of probing actions can apparently improve the accuracy in estimating $\bTheta_o$. Adopting the approximation of \eqref{eq:e_approx} and postulating Gaussian $\be(t)$'s each with covariance matrix $\bSigma$~[cf.~\eqref{eq:Sigma}], the MLE of $\bTheta_o$ can be found to be
\begin{equation}\label{eq:mle}
\hat{\bTheta}_{\text{MLE}}=\bDelta\tbV^\top\left(\tbV\tbV^\top\right)^{-1}.
\end{equation}
Standard analysis ensures that the mean of $\hat{\bTheta}_{\text{MLE}}$ equals $\bTheta_o$, and the covariance of its vectorized form is
\begin{equation}\label{eq:cov}
\mathrm{Cov}[\vectorize(\hat{\bTheta}_{\text{MLE}})]=(\tbV\tbV^\top)^{-1}\otimes \bSigma
\end{equation}
where $\otimes$ denotes the Kronecker matrix product. Moreover, the estimate becomes Gaussian asymptotically in $T$; see~\cite[Ch.~7]{kay93book}. As expected, increasing $T$ beyond $C$ increases $\trace(\tbV\tbV^\top)$, which reduces the trace of the covariance in \eqref{eq:cov}, hence improving the estimation accuracy. The numerical tests of Section~\ref{sec:tests} show how the results improve with increasing $T$. Although the MLE of \eqref{eq:mle} is amenable to a statistical characterization, it relies on both grid approximations and statistical assumptions, and ignores the fact that $\bTheta_o\in \mcM\cap \mcT$. For example, matrix $\hat{\bTheta}_{\text{MLE}}$ is dense almost surely, even though $\bTheta_o$ is sparse. The estimator in \eqref{eq:id} on the other hand exploits the underlying structure, yet complicates any statistical characterization. 

\begin{figure}[t]
\centering
\includegraphics[width=0.23\textwidth]{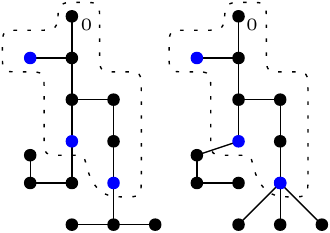}
\caption{The blue (black) nodes represent probing (non-probing) nodes. The left panel depicts the actual topology, and the right panel depicts the recovered topology. The portion of the network within the dashed line can still be perfectly recovered.}
\label{fig:Rem3}
\end{figure}

\subsection{Convex relaxation}\label{subsec:relaxid}
Albeit its objective and set $\mcM$ are convex, the optimization in \eqref{eq:id} is challenging because $\mcT$ is non-convex and open. To arrive at a practical solution, we surrogate $\mcT$ by adding two penalties in the objective of \eqref{eq:id} as detailed next. The property $\bTheta\succ \mathbf{0}$ of $\mcT$ is equivalent to enforcing a finite lower bound on $\log |\bTheta|$. Upon dualizing the two constraints comprising $\mcT$, \eqref{eq:id} can be written in its Lagrangian form as
\begin{equation}\label{eq:id2}
\min_{\bTheta\in \mcM}~f(\bTheta) + \lambda_0\|\bPhi(\bTheta)\|_{0,\text{off}}-\mu_0 \log |\bTheta|
\end{equation}
for some $\lambda_0,\mu_0>0$. The last term in the cost of \eqref{eq:id2} is convex and guarantees that the minimizer is strictly positive definite. The second term though remains non-convex. Adopting the idea of compressive sampling, the non-convex pseudo-norm $\|\bPhi(\bTheta)\|_{0,\text{off}}$ will be surrogated by its convex envelope $\|\bPhi(\bTheta)\|_{1,\text{off}}:=\sum_{m,n\neq m} \left|[\bPhi(\bTheta)]_{mn}\right|$; see also \cite{GLasso}, \cite{LiPoSc13}, \cite{KGB14}, and \cite{egilmez2016graph} for related approaches aiming to recover sparse inverse covariance or Laplacian matrices. The convex function $\|\bPhi(\bTheta)\|_{1,\text{off}}$ can be rewritten as
\begin{align*}
\|\bPhi(\bTheta)\|_{1,\text{off}}&=\trace \left(\bPhi(\bTheta)(\bI-\mathbf 1 \mathbf 1^\top)\right)\\
&= \trace\left(\bTheta(\bI-\mathbf 1 \mathbf 1^\top)\right) + 2 \1^\top \bTheta \1 \\
& = \trace\left(\bTheta(\bI-\mathbf 1 \mathbf 1^\top)\right) + 2 \trace\left(\bTheta \mathbf 1 \mathbf 1^\top\right)\\
& = \trace( \bTheta \bPi)
\end{align*}
where $\bPi:=\bI+\mathbf 1 \mathbf 1^\top$. The first equality follows from the definition of the norm and because the off-diagonal entries of $\bPhi(\bTheta)$ are non-positive [property (p3)]. The second equality from the definition of $\bPhi(\bTheta)$ in \eqref{eq:Phi}, and the third one from properties of the trace. Based on the previous discussion, the non-convex problem in \eqref{eq:id2} is surrogated by the convex
\begin{equation}\label{eq:id3}
\hbTheta:=\arg\min_{\bTheta \in \mcM} ~ \frac{1}{2} \|\bTheta \tbV - \bI_{\mcC} \bDelta\|_\bW^2 + \lambda \trace (\bTheta \bPi) - \mu \log |\bTheta|
\end{equation}
where $\lambda,\mu>0$ are tunable parameters. Since the minimizer of \eqref{eq:id3} does not necessarily belong to $\mcT$, one may apply heuristics to convert it to the reduced Laplacian of a tree graph. As suggested in~\cite{BoSch13}, one may find a minimum spanning tree for the weighted graph defined by $\bPhi(\hbTheta)$ using Kruskal's or Prim's algorithm~\cite{Ahuja}. A Laplacian $\tbTheta$ belonging to $\mcT$ can hence be found by keeping the entries of $\hbTheta$ associated with the edges found by the minimum spanning tree algorithm. Although the convex solver of \eqref{eq:id3} can run even when not all leaf nodes are probed, voltage data still need to be collected at all buses. Either way, studying the success of \eqref{eq:id3} in recovering the actual topology will not be pursued here.

\subsection{Topology identification algorithm}\label{sub:top_alg}
Albeit convex, the problem in \eqref{eq:id3} cannot be solved directly by standard conic optimization solvers due to its last term. For this reason, we choose to tackle \eqref{eq:id3} using ADMM. As a brief review, ADMM solves problems of the form~\cite{Be15}
\begin{subequations}\label{eq:ADMM}
\begin{align}
\min_{\bx \in \mcX,  \bz \in \mcZ} ~&~  h(\bx) + g(\bz)\label{eq:ADMM:cost}\\
\textrm{s.t.} ~&~\bF \bx + \bG \bz = \bc \label{eq:ADMM:con}
\end{align}
\end{subequations}
where $h(\bx)$ and $g(\bz)$ are convex functions; $\mcX$ and $\mcZ$ are convex sets; and $(\bF, \bG, \bc)$ are matrices/vectors of compatible dimensions coupling linearly $\bx$ and $\bz$. In its normalized form, ADMM solves the problem in \eqref{eq:ADMM} by iteratively repeating the next three steps for a step size $\rho>0$
\begin{align*}
&\bx^{k+1} = \arg \min_{\bx \in \mcX} ~ h(\bx) + \frac \rho 2 \| \bF \bx + \bG \bx^{k} - \bc  + \bm^{k} \|^2_2\\
&\bz^{k+1} = \arg \min_{\bz \in } ~ g(\bx) + \frac \rho 2 \| \bF \bx^{k+1} + \bG \bz - \bc  + \bm^{k} \|^2_2\\
&\bm^{k+1} = \bm^{k} + \bF \bx^{k+1} + \bG \bz^{k+1} - \bc
\end{align*}
where $\bm$ is the Lagrange multiplier vector corresponding to the linear equality constraint of \eqref{eq:ADMM:con}.

To exploit ADMM, task \eqref{eq:id3} is equivalently expressed as
\begin{subequations}\label{eq:idADMM}
\begin{align}
\min ~&~ \frac{1}{2} \|\bTheta_1 \tbV - \bI_{\mcC} \bDelta \|_\bW^2 +  \lambda \trace (\bTheta_1 \bPi) - \mu \log |\bTheta_2| \label{eq:idADMM_cost}\\
\textrm{over}~&~\bTheta_1, \bTheta_2\succeq \bZero, \bTheta_3, \bTheta_4\\
\textrm{s.t.} ~&~ \bTheta_3 \in \mathcal S(\mathbf \Gamma) \\
~&~ \bTheta_4 \in \mathcal S_0(\mathbf \Gamma) \\
~&~ \bTheta_1 = \bTheta_2,~ \bTheta_1 = \bTheta_3, ~ \bTheta_1 = \bTheta_4 \label{eq:idADMM_con4}
\end{align}
\end{subequations}
where the original variable $\bTheta$ is replicated in four copies $\{\bTheta_i\}_{i=1}^4$, each one handling a different constraint or part of the cost function. This replication was necessary in order to yield efficient updates for each copy as will be described later. The variables in \eqref{eq:idADMM} are partitioned into $\bTheta_1$ which is mapped to the $\bx$ variable in the general ADMM form, and $\{\bTheta_2,\bTheta_3,\bTheta_4\}$ which correspond to the $\bz$ ADMM variable. Let $\bM_{2},\bM_3$, and $\bM_4$ be the Lagrange multipliers associated with the constraints in \eqref{eq:idADMM_con4}. 

Then, the $\bx$-update of ADMM for finding $\bTheta_1^{k+1}$ entails minimizing the convex quadratic cost $\frac{1}{2} \|\bTheta_1 \tbV - \bI_{\mcC} \bDelta \|_\bW^2 + \lambda \trace (\bTheta_1 \bPi) + \frac \rho 6 \|3 \bTheta_1 - \bTheta_2^k - \bTheta_3^k - \bTheta_4^k + \bM_{2}^{k} + \bM_3^{k} + \bM_4^{k} \|^2_F$. Its solution can be found by simply setting its gradient to zero to get the system of linear equations
\begin{equation}\label{eq:sylvester1}
\bW\bTheta_1^{k+1} \tbV\tbV^\top + 3 \rho \bTheta_1^{k+1} = \bC^k
\end{equation}
where $\bC^k:=\bW \bI_{\mcC} \bDelta\tbV^\top - \lambda  \bPi -\rho(\bM_{2}^k + \bM_3^k + \bM_4^k - \bTheta_2^{k} - \bTheta_3^{k} - \bTheta_4^{k})$. Since $\bW$ is full-rank, the solution to \eqref{eq:sylvester1} can be efficiently found by solving the Sylvester equation~\cite[Ch.~2.4.4]{horn1990matrix}
\begin{equation*}
\bTheta_1^{k+1} \tbV \tbV^\top + 3 \rho \bW^{-1} \bTheta_1^{k+1}  = \bW^{-1}\bC^k.
\end{equation*}

The $\bz$-update of ADMM computes $\{\bTheta_2^{k+1},\bTheta_3^{k+1},\bTheta_4^{k+1}\}$ given $\bTheta_1^{k+1}$ and the multipliers $\{\bM_2^k,\bM_3^k,\bM_4^k\}$. Due to the way the variable copies have been selected and partitioned, the optimization involved in the $\bz$-update step of ADMM decouples over the three primal variables. In particular, variable $\bTheta_2$ is updated by solving the problem
\begin{equation*}
\bTheta_2^{k+1}  := \arg \min_{\bTheta_2\succeq \mathbf{0}}~ \frac{\rho}{2} \| \bTheta_1^{k+1} - \bTheta_2 + \bM_3^{k} \|_F^2  - \mu \log |\bTheta_2|
\end{equation*}
whose minimizer can be found in closed form as follows: If $\mathbf U\mathbf \Lambda\mathbf U^\top$ is the eigenvalue decomposition of $(\bTheta_1^{k+1} + \bM_3^{k} + (\bTheta_1^{k+1})^\top + (\bM_3^{k})^\top)/2$, then~\cite[Lemma~1]{KGB14}
\begin{equation*}
\bTheta_2^{k+1} = \frac{1}{2} \mathbf U \left( \mathbf \Lambda + \left(\mathbf \Lambda^2 +  \tfrac {4\mu}{\rho} \bI \right)^{1/2}\right) \mathbf U^\top.
\end{equation*}

Variable $\bTheta_3^{k+1}$ is obtained as the minimizer of 
\begin{equation}\label{eq:ADMM_s2}
\bTheta_3^{k+1}  :=  \arg \min_{\bTheta_3 \in \mathcal S(\mathbf \Gamma)}  \| \bTheta_1^{k+1} - \bTheta_3 + \bM_{2}^{k} \|_F^2
\end{equation}
which can be expressed in closed form as
\begin{equation*}
[\Theta_3]_{mn}^{k+1}  =  \left\{
\begin{array}{ll}
0&\text{,~if~}\Gamma_{mn} = 0 \\
\left[ [\Theta_1]_{mn}^{k+1} + [M_2]_{mn}^{k}\right]_{-}& \text{,~if~} \Gamma_{mn} \neq 0, m \neq n\\
{\Theta_1}_{mn}^{k+1} + M_{2,mn}^{k} & \text{,~if~}m = n
\end{array} \right.
\end{equation*}
where operator $[x]_{-}$ is defined as $[x]_{-}:=\min\{x,0\}$.

Variable $\bTheta_4$ can be computed as the solution of
\begin{align}\label{eq:s4}
\bTheta_4^{k+1}  := \arg \min_{\bTheta_4} & ~  \| \bTheta_1^{k+1} - \bTheta_4 + \bM_4^{k} \|_F^2 \\
\textrm{s.t.} & ~ \bTheta_4 \in \mcS_0(\mathbf \Gamma).\nonumber
\end{align}
The minimization in \eqref{eq:s4} is separable over the rows of $\bTheta_4$. Each row can be updated in closed form as asserted by the next lemma, which follows from the related optimality conditions.

\begin{lemma}\label{le:s4}
The projections of vector $\by\in\mathbb{R}^N$ onto the halfspace $\1^\top\mathbf{x}\geq 0$ and the subspace $\1^\top\mathbf{x} = 0$, that is
\begin{align*}
&\hat{\bx}_h:=\arg\min_{\bx}\left\{\|\bx-\by\|_2^2:~\1^\top\mathbf{x}\geq 0\right\}\\
&\hat{\bx}_s:=\arg\min_{\bx}\left\{\|\bx-\by\|_2^2:~\1^\top\mathbf{x} = 0\right\}
\end{align*}
admit respectively the closed-form solutions 
\begin{align*}
&\hat{\bx}_h =  \by -\frac{1}{N}\1 \left[\1^\top \by\right]_{-}~\text{and}~
&\hat{\bx}_s = \by -\frac{1}{N}\1 \left[\1^\top \by\right].
\end{align*}
\end{lemma}

Finally, the Lagrange multipliers are updated according to
\begin{align}
&\bM_{2}^{k+1} = \bM_{2}^{k} + \bTheta_1^{k+1} - \bTheta_2^{k+1} \notag\\
&\bM_3^{k+1} = \bM_3^{k} + \bTheta_1^{k+1} - \bTheta_3^{k+1} \notag\\
&\bM_4^{k+1} = \bM_4^{k} + \bTheta_1^{k+1} - \bTheta_4^{k+1}. \label{eq:ADMM_s5_cf}
\end{align}

\section{Topology Verification}\label{sec:verification}
Albeit typically operated as radial, power distribution networks are structurally meshed for reliability and maintenance. Moreover, grids are frequently reconfigured to improve voltage profiles and/or minimize losses~\cite{BW3}. The energized lines $\mcL$ are chosen from the set of existing lines denoted as $\mcL_e$ with $\mcL \subset \mcL_e$ and $|\mcL_e|=L_e$. The actual grid configuration is oftentimes unknown to the system operator, e.g., due to unreported automatic reconfigurations \cite{TalDeMate2017}. This section explores how grid probing can be used for topology verification. 

\subsection{Problem formulation}\label{subsec:problem-ver}
Given probing sequences, the recorded voltages, the existing line set $\mcL$, and line impedances, the goal of topology verification is to find which lines are energized. A line configuration can be uniquely encoded by an $L_e$--length vector $\bb$, whose entry $b_\ell$ is one if line $\ell$ is energized ($\ell \in \mcL$); and zero, otherwise. Given the line configuration $\bb$, the induced ancestor, descendant, and level sets of node $m$ will be denoted as $\mcA_m(\bb)$, $\mcD_m(\bb)$, and $\{\mcN_m^k(\bb)\}_{k=1}^{d_m}$, respectively. In addition, the related reduced Laplacian matrix can be written as
\begin{equation}\label{eq:theta(b)}
\bTheta(\bb) = \bA^\top \diag(\bb)\diag^{-1}(\br)\bA
\end{equation}
where $\bA\in \mathbb{R}_{L_e\times N}$ is the reduced incidence matrix augmented to incorporate all lines in $\mcL_e$.

Upon probing the grid with active power injections $\bI_{\mcC} \bDelta$ and recording the voltage perturbations $\tbV$, the task of topology verification can be posed as
\begin{subequations}\label{eq:tv}
\begin{align}
\hat \bb := \arg \min_\bb ~&~ \|\bTheta(\bb) \tbV - \bI_{\mcC} \bDelta\|_\bW^2 \label{eq:tv_cost} \\
\textrm{s.t.} ~&~ \bb \in \{0,1\}^{L_e} \label{eq:tv_c1}\\
~&~ \mathbf{1}^\top \bb=N \label{eq:tv_c2}\\
~&~ \bTheta(\bb) \succ \mathbf 0 \label{eq:tv_c3}.
\end{align}
\end{subequations}
Under the assumptions described in \eqref{eq:e_approx}--\eqref{eq:Sigma}, the minimizer of \eqref{eq:tv} is in fact the maximum likelihood detector of the true line status vector $\bb_o$. In a grid with $N+1$ nodes, the constraints in \eqref{eq:tv} enforce radial structure. To see this, note that the graph induced by any $N$ edges not forming a tree is not a connected graph. The corresponding reduced incidence matrix $\bA(\bb)$ is column-rank deficient, and so $\bTheta(\bb)$ becomes singular. Contrasting \eqref{eq:id} to \eqref{eq:tv}, the latter problem recovers line statuses while knowing resistances, whereas \eqref{eq:id} recovers jointly the topology and line resistances.

\subsection{Verifiability analysis}\label{subsec:va}
A pertinent question is whether probing the grid via a set of buses $\mcC$, its topology can be uniquely recovered by \eqref{eq:tv}.

\begin{corollary}\label{pro:leaf+ver}
Given probing data $(\tbV,\bDelta)$ where $\tbV = \bR(\bb_o) \bI_{\mcC} \bDelta$ and $\textrm{rank}(\bDelta)=C$, the grid topology is verifiable by solving the problem in \eqref{eq:tv} if the grid is probed at all leaf nodes, that is $\mcF(\bb_o) \subseteq \mathcal C$.
\end{corollary}

Corollary~\ref{pro:leaf+ver} follows from Theorem~\ref{th:leaf+id}, and establishes that if the grid is probed at all candidate leaf nodes, its topology is verifiable. If $\mcF(\bb_o) \nsubseteq \mcC$, problem~\eqref{eq:tv} may have multiple minimizers. Nonetheless, if additional \emph{a priori} knowledge is exploited, the topology may still be identifiable as quantified by the following result, built upon the next mild condition.

\begin{assumption}\label{ass:diff_resistances}
All lines $\ell\in \mcL_e$ have distinct resistances. 
\end{assumption}

\begin{theorem}\label{th:probe}
Under Assumption~\ref{ass:diff_resistances}, given noiseless probing data $(\tbV,\bDelta)$ where $\tbV = \bR(\bb_o) \bI_{\mcC} \bDelta$ and $\textrm{rank}(\bDelta)=C$, a line status vector $\bb$ is a minimizer of \eqref{eq:tv} if and only if it satisfies 
\begin{subequations}\label{eq:probe}
\begin{align}
&\mcA_m(\bb) = \mcA_m(\bb_o) \label{eq:probe:a}\\
&\mcN^k_m(\bb) = \mcN^k_m(\bb_o), \quad k=0,\ldots,d_m\label{eq:probe:b}
\end{align}
\end{subequations}
for every probing bus $m\in \mcC$.
\end{theorem}

\begin{figure}[t]
\centering
\includegraphics[width=0.4\textwidth]{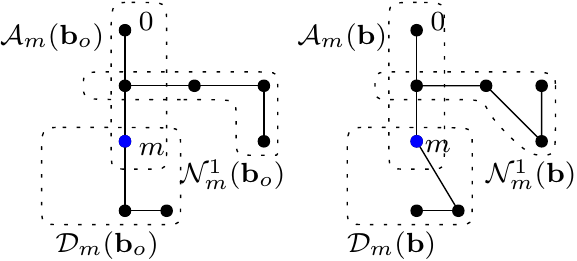}
\caption{Let $m$ be the only probing bus. The left panel shows the actual grid topology, while the right panel depicts one of the possible minimizers of \eqref{eq:tv}. The path between $m$ and the root $0$ is the same in both configurations. The topologies differ only in the connections among buses in $\mcD_m$ or $\mcN_m^1$.}\label{fig:prop3a}
\end{figure}

\begin{figure}[t]
\centering
\includegraphics[width=0.38\textwidth]{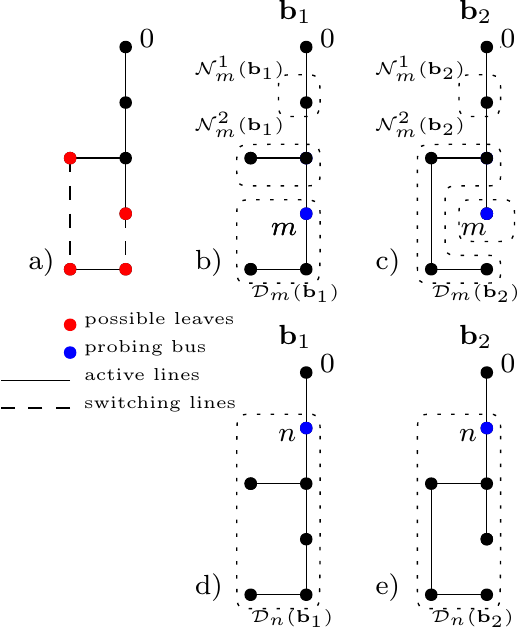}
\caption{Pictorial representation of Example~\ref{ex:verifibiality}.}\label{fig:prop3b}
\end{figure}

Theorem~\ref{th:probe} implies that every solution of \eqref{eq:tv} is such that the paths between the probing nodes and the substation, and the level sets are correctly recovered. See Figure~\ref{fig:prop3a} for a pictorial explanation and compare it with Fig.~\ref{fig:LemmaLS}. Knowing the line infrastructure can help reconstructing the unknown grid topology even if not all leaf nodes are probed. For instance, let $m$ be a non-leaf probing node. It may happen that, given a line status $\bb$, the connections among nodes in $\mcN^k_m(\bb)$ can be uniquely determined, e.g., when there are no switching lines connecting nodes in such level set. As a result, the number of probing buses can be dramatically reduced, as shown in the ensuing example.

\begin{example}
\label{ex:verifibiality}
Consider the grid of Fig.~\ref{fig:prop3b}a) having two possible radial configurations, $\bb_1$ and $\bb_2$. Corollary~\ref{pro:leaf+ver} guarantees verifiability if all four possible leaf nodes are probed. However, the topology can be detected even if the grid is probed only at bus $m$: Assume $\bb_1$ is the true topology. Fig.~\ref{fig:prop3b}b) and Fig.~\ref{fig:prop3b}c) show that $\mcD_m(\bb_1) \neq \mcD_m(\bb_2)$ and $\mcN^2_m(\bb_1) \neq \mcN^2_m(\bb_2)$. Then, Th.~\ref{th:probe} states that $f(\bb_1) \neq f(\bb_2)$, and thus $\bb_1$ is the unique solution of \eqref{eq:tv}. On the other hand, when $n$ is the only probing bus, the network is not verifiable: Fig.~\ref{fig:prop3b}d) and Fig.~\ref{fig:prop3b}e) show that $\mcD_n(\bb_1) = \mcD_n(\bb_2)$. Hence, from Th.~\ref{th:probe} it holds that $f(\bb_1) = f(\bb_2)$, so $\bb_1 $ and $\bb_2$ are indistinguishable.
\end{example}

\subsection{Verification algorithms}\label{sub:ver_alg}
Solving \eqref{eq:tv} is non-trivial since it is a non-convex problem. An approximate line status vector can be found by considering the surrogate convex problem
\begin{align}\label{eq:tv_rel}
\cbb:=\arg\min_\bb ~& ~ \tfrac 1 2 \| \bTheta(\bb) \tbV - \bI_{\mcC} \bDelta \|_\bW^2 - \mu \log |\bTheta(\bb)|\\
\textrm{s.t.} ~&~\bb \in [0,1]^{N},~\mathbf{1}^\top \bb=N.\nonumber
\end{align}
Comparing the optimization problems in \eqref{eq:tv} and \eqref{eq:tv_rel}, the set $\{0,1\}^{N}$ has been relaxed by its convex hull. Moreover, the constraint \eqref{eq:tv_c3} was substituted by the $- \mu \log |\bTheta(\bb)|$ term in \eqref{eq:tv_rel}. The latter acts as a barrier function keeping $\bTheta(\bb)$ within the positive definite matrix cone and away from singularity.

Since \eqref{eq:tv_rel} does not comply with the standard form of a conic problem, it can be handled by a projected gradient descent (PGD) scheme. Upon initializing $\bb$ at some $\bb^0$, the $k$-th PGD iteration reads
\begin{align}\label{eq:PGD}
\bb^{k+1}:=\arg\min_{\bb}&~\|\bb-\bb^k+\nu \bg(\bb^k)\|_2^2\\
\textrm{s.t.}&~ \bb \in [0,1]^{N},~\mathbf{1}^\top \bb=N\nonumber
\end{align}
where $\nu>0$ is a step size and $\bg(\bb^k)$ is the gradient of the objective in \eqref{eq:tv_rel} evaluated at $\bb^k$. Using rules of matrix differentiation, the $\ell$-th entry of $\bg(\bb)$ is computed as
\begin{equation*}
[\bg(\bb)]_{\ell}=\ba_\ell^\top \left[\tbV  (\bTheta(\bb) \tbV - \bI_{\mcC} \bDelta)^\top \bW - \mu \bTheta^{-1}(\bb)\right] \ba_\ell/r_\ell
\end{equation*}
where $\ba_\ell^\top$ is the $\ell$-th row of $\bA$. The iterations in \eqref{eq:PGD} are guaranteed to converge to a minimizer of \eqref{eq:tv_rel} for a sufficiently small $\nu$~\cite[Prop. 6.1.3]{Be15}. The projection step in \eqref{eq:PGD} can be handled either as a generic linearly-constrained convex quadratic program; by the lambda iteration method~\cite[Sec.~5.2.4]{ExpConCanBook}; or by dual ascent upon dualizing the constraint $\mathbf{1}^\top \bb=N$. However, the minimizer of \eqref{eq:tv_rel} may not lie in the original non-convex set $\{0,1\}^N$. A feasible vector $\tbb$ can be heuristically obtained either by selecting the lines corresponding to the largest $N$ entries of $\cbb$, or by finding the minimum spanning tree on a graph having $\cbb$ as edge weights.

\section{Numerical Tests}\label{sec:tests}
The novel feeder processing schemes were validated on the IEEE 13- and 37-bus feeders~\cite{Kersting}, converted to their single-phase equivalents~\cite{CaKe2017}. Additional lines were added to both feeders as shown in Fig.~\ref{fig:ieee37}, so that the 13-bus (37-bus) testbed could operate under 7 (21) distinct radial topologies. Regarding loads, active load profiles were generated by adding a zero-mean Gaussian-distributed variation to the benchmark value; its standard deviation was 0.067 times the load nominal value. For simplicity, a lagging power factor of 0.95 was simulated, and no prior information on line statuses was assumed $(\tbGamma = \mathbf 1 \mathbf 1^\top)$.

Probing and data collection were performed on a per-second basis. Since our approaches operate in a batch manner, the voltage readings from each bus can be communicated all together after all $T$ probing actions. Probing buses were equipped with inverters having the same rating as the related load. Unless otherwise stated, the buses used for probing are indicated in Fig.~\ref{fig:ieee37}. Although our approaches rely on the approximate grid model of \eqref{eq:model}, voltages were calculated using the full ac grid model throughout our tests. Probing actions was performed \emph{asynchronously} using the probing matrix $\bDelta_2= \diag(\{\delta_m\})\otimes [+1~-1]$ explained after \eqref{eq:dV}.

\begin{figure}[t]
\centering	
\includegraphics[width=0.27\textwidth]{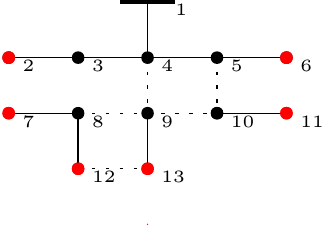}
\includegraphics[width=0.38\textwidth]{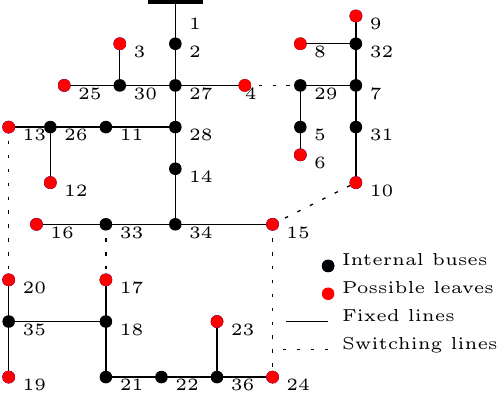}
\caption{The IEEE 13-bus (top) and 37-bus (bottom) feeders with extra lines.}
\label{fig:ieee37}
\end{figure}
 
The topology identification task of \eqref{eq:id3} was solved for different levels of measurement noise and values of $\lambda$ on the 13 IEEE bus test feeder, in the case in which each agent perform a single probing action. Upon experimentation, the remaining parameters were set to $\rho = 1$, $\mu=1$, and $\bW = \bI$. Actual voltage magnitudes were corrupted by zero-mean Gaussian noise whose $3\sigma$ deviation matched the desired level of relative measurement noise: For example, if the actual voltage is 1 pu and the accuracy is 1\%, the maximum value of noise is 1.01 and this value was mapped to the $3\sigma$ deviation of the Gaussian noise. Tree Laplacian matrices $\tbTheta$ were obtained by feeding $\hbTheta$ to Kruskal's algorithm. Figure~\ref{fig:emse} reports the relative mean square error (RMSE) of $\tbTheta$, that is $\frac{\| \tbTheta - \bTheta_o\|_F}{\|\bTheta_o\|_F}$, for varying noise variances. As expected, the RMSE increases with the noise variance, while the best performance is attained for $\lambda = 0.005$. Figure~\ref{fig:lapl} depicts the actual Laplacian matrix $\bTheta_o$ and the recovered $\tbTheta$ for $\lambda= 0.005$.

\begin{figure}[t]
\centering	
\includegraphics[width=0.5\textwidth]{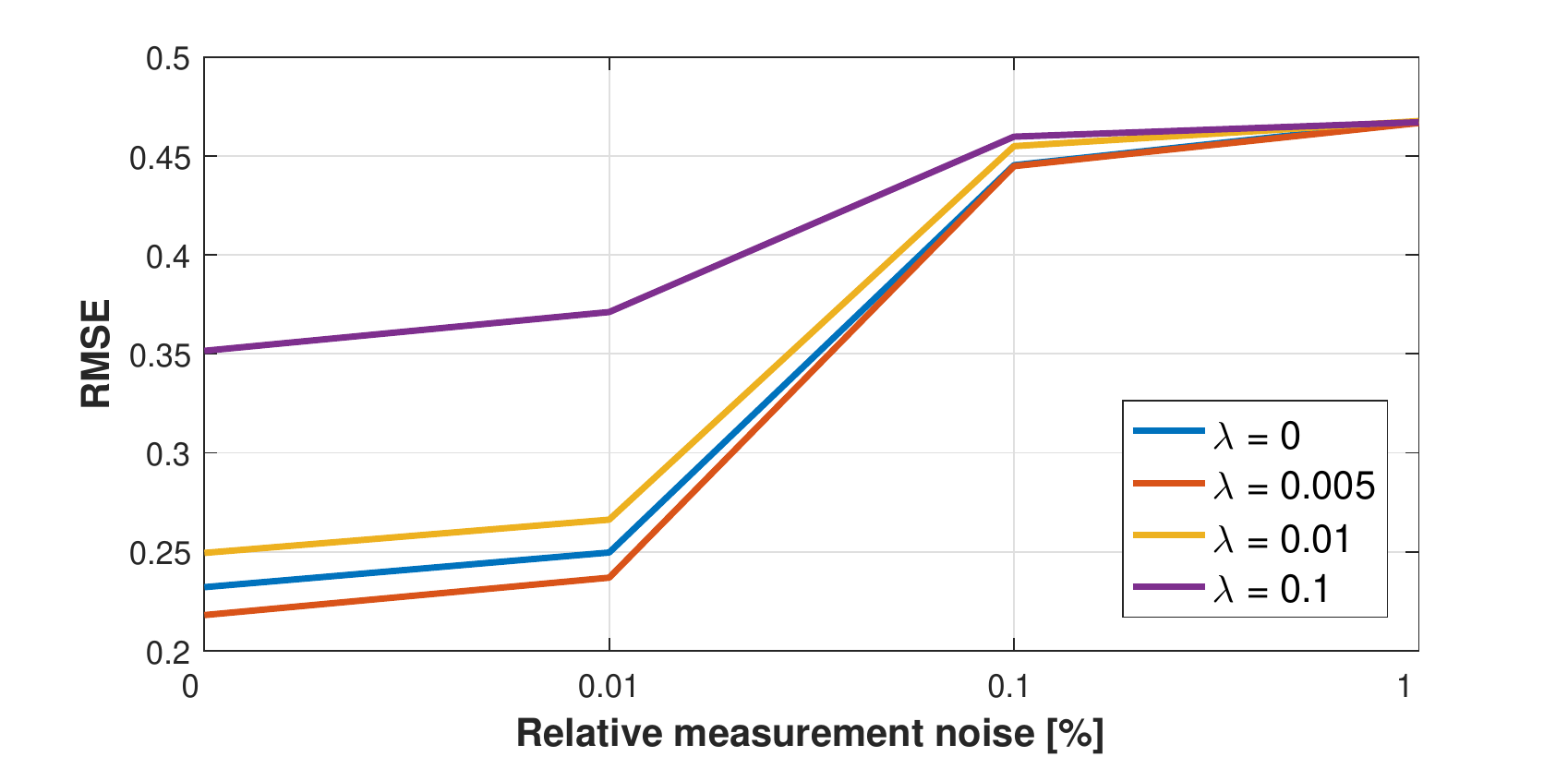}
\caption{Relative mean squared error (RMSE) for $\tbTheta$.}
\label{fig:emse}
\end{figure}

\begin{figure}[t]
\centering	
\includegraphics[width=0.47\textwidth]{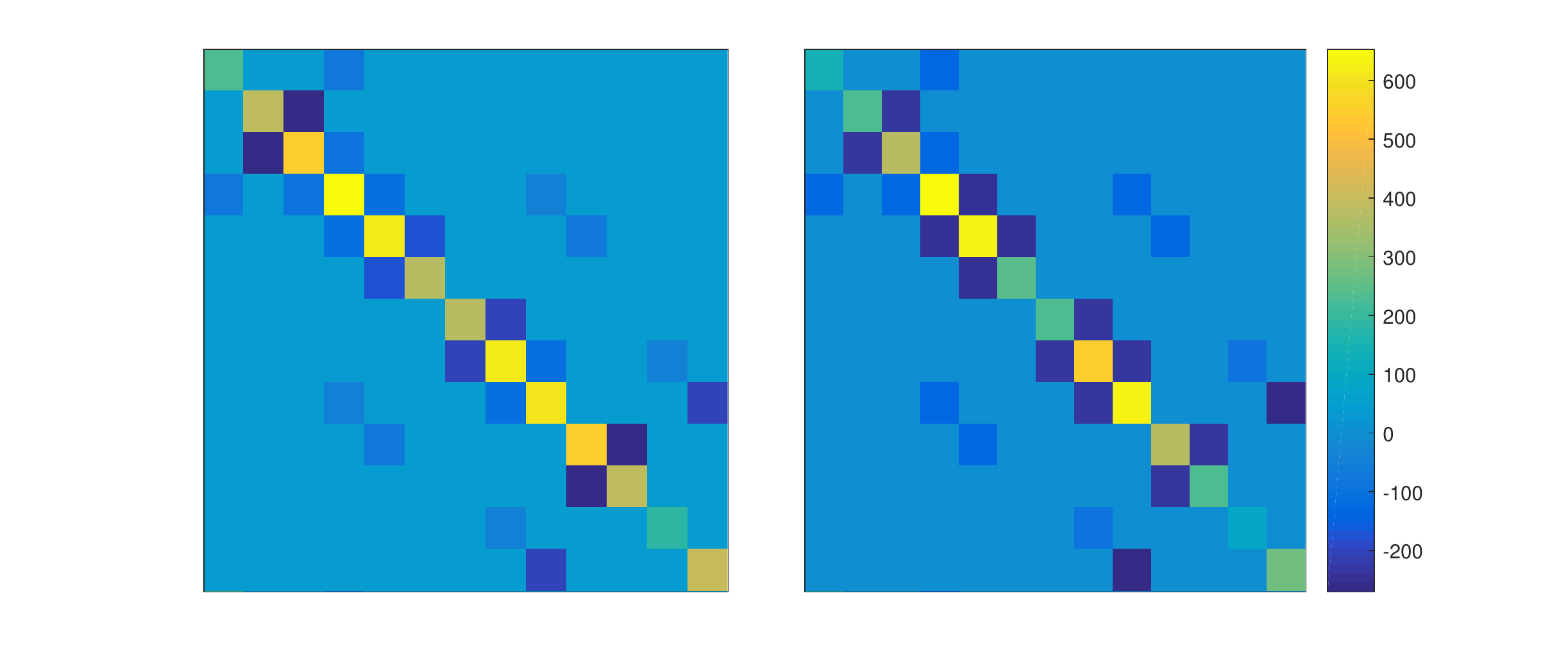}
\caption{The actual (recovered) Laplacian matrix is shown on the left (right).}
\label{fig:lapl}
\end{figure}

The identification and verification schemes were also evaluated on the IEEE 37 bus testbed using 200 Monte Carlo runs for a relative measurement noise of 0.01\%. At every run, the actual topology was randomly drawn. The identification problem in \eqref{eq:idADMM} was solved for $\rho = 1$, $\lambda = 5 \cdot 10^{-3}$, $\mu = 1$, and $\bW = \bI $. Kruskal's algorithm was used to obtain a Laplacian matrix $\tbTheta$ corresponding to a radial grid. The verification problem of \eqref{eq:tv_rel} was solved for $\mu=2 \cdot 10^{-8}$, $\nu=10^{-10}$, and $\bW = \bI$. A feasible vector $\tbb$ was recovered by selecting the lines corresponding to the $N$ largest entries of $\cbb$. 

Given no prior information on line status, the identification task had to choose among all $(36 \cdot 35)/2 = 630$ possible lines. On the other hand, the verification task had to choose among 38 possible lines. The connectivity of $\tbTheta$ and $\bTheta(\tbb)$ were compared against the actual ones. The average number of line status errors are reported in Table~\ref{tbl:linedata}. The errors here include both the energized lines not detected and the non-energized lines detected. The tests show that topology recovery requires approximately $16$ to $160$ $1$-sec probing actions. This yields a topology processing time of less than $160$~sec, which is significantly reduced over the passively collected data schemes of \cite{BoSch13}, \cite{Deka4}, \cite{ParkDeka}, \cite{WengLiaoRajagopal17}, requiring thousands of data samples.

\begin{table}[t]
\renewcommand{\arraystretch}{1.2}
\caption{Average Number of Line Status Errors for the 37-bus feeder}
\label{tbl:linedata} \centering
\begin{tabular}{|l|c|c|c|c|}
\hline
\hline
 & $T = 1$ & $T = 2$ & $T = 5$ & $T = 10$ \\
\hline
\hline
Identification task of \eqref{eq:id3} & 5.07 & 3.92 & 3.73 & 2.69\\
Verification task of \eqref{eq:tv_rel} & 0.32 & 0.21 & 0.08 & 0.01\\
\hline
\hline
\end{tabular}
\end{table}

Finally, the extreme scenario where only two buses perform a single probing action was tested. Figure~\ref{fig:desc_comp} shows the identification and verification results for probing buses $\mcC = \{14,17\}$ and $\mcC = \{4,14\}$. Interestingly, the verification algorithm manages to estimate correctly the line connectivity under the first setup. 

\begin{figure}[t]
\centering	
\includegraphics[width=0.40\textwidth]{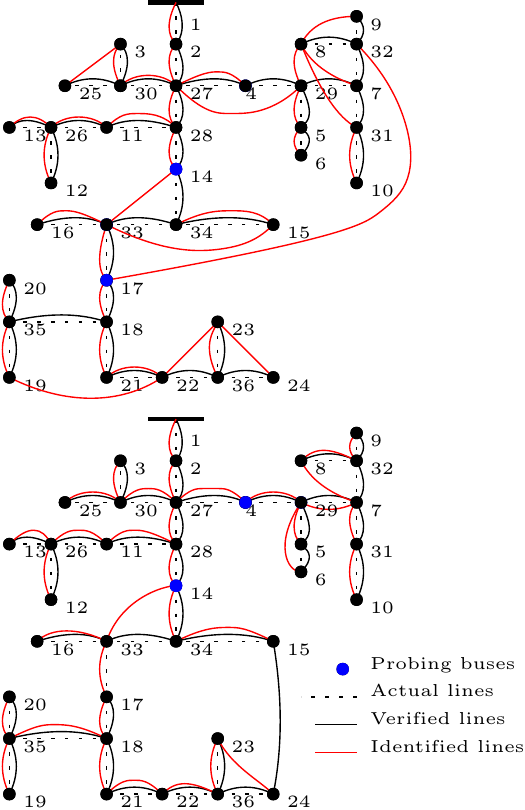}
\caption{Topology identification and verification for the IEEE 37-bus feeder upon probing buses $\{14,17\}$ (top) and $\{4,14\}$ (bottom).}
\label{fig:desc_comp}
\end{figure}

\section{Conclusions}\label{sec:conclusions}
The novel idea of perturbing an electric distribution network through smart inverters to actively collect data and infer its topology has been put forth here. Given voltage data collected via grid probing, the tasks of topology identification and verification have been posed as the non-convex minimization tasks of finding the Laplacian matrix of a radial graph. Using the graph-theoretic notion of a level set, probing only terminal nodes has been analytically shown to be sufficient for exact topology recovery. Convex relaxations of the original minimization problems have been efficiently handled through closed-form ADMM updates. Numerical tests on benchmark feeders demonstrated that the novel grid probing schemes can attain line status error probabilities in the order of $10^{-2}-10^{-3}$ by probing 40\% of the buses. 

Upon introducing grid probing for topology processing, this work sets the foundations for exciting research directions. Combining active and reactive probing could jointly recover the resistive and reactive Laplacian matrices. Generalizing our schemes to multiphase grids constitutes a practically pertinent and challenging problem. Processing voltage magnitude or synchrophasor data at a subset of nodes, optimally selecting probing buses and probing injections, and devising real-time data processing schemes, all constitute challenging research directions.

\appendix
\begin{IEEEproof}[Proof of Lemma~\ref{le:levelsets}]
Consider the leaf node $m\in\mcF$ and its $k$-depth ancestor $n= \alpha^k_m$. Select the subset of leaves $\mcW=\mcF \cap \mcD_n$ with $m\in\mcW$. We need to show that $\mcW$ satisfies \eqref{eq:levelsets}:
\begin{align*}
\bigcap_{w\in \mcW}  \mcN_w^k &= \bigcap_{w\in \mcW} \left(\mcD_{\alpha^{k}_w}\setminus \mcD_{\alpha^{k+1}_w} \right)\\
&= \bigcap_{w\in \mcW} \left(\mcD_n \setminus \mcD_{\alpha^{k+1}_w} \right)\\
&=\mcD_n\setminus \bigcup_{w\in \mcW}\mcD_{\alpha^{k+1}_w}\\
& =\{n\}.
\end{align*}
The first equality follows by the definition in \eqref{eq:levelset}; the second one holds since $n=\alpha^k_w$ for all $w\in\mcW$; the third one is an identity for set operations; and the last holds by recognizing that $\bigcup_{w\in \mcW}\mcD_{\alpha^{k+1}_w}=\mcD_n\setminus \{n\}$.

The converse will be shown by contradiction. Suppose there exists a $\mcW\subseteq\mcF$ with $m\in\mcW$ that satisfies \eqref{eq:levelsets} yet $n\neq \alpha^k_m$. Since the only $k$-depth node in $\mcN^k_m$ is $\alpha_m^k$, node $n$ has to have depth greater than $k$, i.e., $d_n > k$. From Lemma~\ref{le:Nmk}, if $n \in \mcN^k_w$, then $\alpha_n^k\in\mcN_w^k$ as well for all $w\in \mcW$. It therefore follows that
\begin{equation}\label{eq:levelsets1}
\{n,\alpha_n^k\} \subseteq  \bigcap \limits_{w \in \mcW} \mcN^k_w.
\end{equation}
Heed that nodes $n$ and $\alpha_n^k$ are distinct because $d_{\alpha_n^k}=k<d_n$. Hence, the intersection in the RHS of \eqref{eq:levelsets1} is not unique contradicting \eqref{eq:levelsets} and proving the claim.
\end{IEEEproof}

\begin{IEEEproof}[Proof of Lemma~\ref{le:uniqueness}]
For the sake of contradiction, assume there exists another tree graph $\mcG' = (\mcN,\mcE')$ with $\mcE'\neq \mcE$ such that $\mcF(\mcG) = \mcF(\mcG')$ and $\{\mcN_w^k(\mcG) = \mcN_w^k(\mcG')\}_{k=0}^{d_w}$ for all $w \in \mcF(\mcG)$. Note that $d_w(\mcG)=d_w(\mcG')$ for all $w\in\mcF$ since by hypothesis the level sets of $w$ agree across the two graphs. 

The notation $\mcT_n(\mcG)$ represents the vertex-induced subtree created by maintaining only the nodes in $\mcD_n(\mcG)$ along with their incident edges. Since $\mcG' \neq \mcG$, there exists a subtree $\mcT_n(\mcG')$ with the properties: (i) it appears also in the original graph, that is $\mcT_n(\mcG')=\mcT_n(\mcG)$; and (ii) the line $e'=(m,n)\in\mcE'$ feeding the node $n$ in $\mcG'$ does not occur in the original graph ($e'\notin \mcE$). Such a $\mcT_n(\mcG')$ exists and it may be the singleton $\mcT_n(\mcG')=\{n\}$ for any leaf node $n\in \mcF(\mcG)=\mcF(\mcG')$.

Let the depths of nodes $m$ and $n$ in $\mcG'$ be $d_m(\mcG')=k-1$ and $d_n(\mcG')=k$. Then, it holds that
\begin{equation}\label{eq:mparentn}
m = \alpha^{k-1}_n(\mcG'), ~~\text{but}~~ m \neq \alpha^{k-1}_n(\mcG).
\end{equation}

We will next show that $d_m(\mcG')\neq d_m(\mcG)$. To do so, consider a leaf node $s\in \mcT_n(\mcG')$ without excluding the case $s=n$. From \eqref{eq:mparentn}, we have that node $m$ is also the $(k-1)$-depth ancestor of $s$ in $\mcG'$ but not in $\mcG$, that is
\begin{subequations}\label{eq:mparentw}
\begin{align}
m &= \alpha^{k-1}_s(\mcG')\label{eq:mparentw:a}\\
m &\neq \alpha^{k-1}_s(\mcG).\label{eq:mparentw:b}
\end{align}
\end{subequations}
From Lemma~\ref{le:Nmk} and \eqref{eq:mparentw:a}, we get that $m\in \mcN_s^{k-1}(\mcG')$. Equation~\eqref{eq:mparentw:b} on the other hand along with the hypothesis $\mcN_s^{k-1}(\mcG') = \mcN_s^{k-1}(\mcG)$ imply that $d_m(\mcG)> k-1=d_m(\mcG')$. Hence, node $m$ has different depths in $\mcG$ and $\mcG'$.

Consider now the set $\mcW:= \mcD_m(\mcG') \cap \mcF(\mcG')$. By definition of $\mcW$, it holds that $m=\alpha_w^{k-1}(\mcG')$ for all $w \in \mcW$. Invoking Corollary~\ref{co:Wexample}, it follows that
\begin{equation}\label{eq:inter1}
\{m\}= \bigcap \limits_{w \in \mcW} \mcN^{k-1}_w(\mcG').
\end{equation}
However, in $\mcG$ it holds that $m \neq \alpha_w^{k-1}(\mcG)$ for any $w \in \mcW$ since $d_m(\mcG)>k-1$. Hence, Lemma~\ref{le:Nmk} provides that
\begin{equation}\label{eq:inter2}
\{m\} \neq \bigcap \limits_{w \in \mcW} \mcN^{k-1}_w(\mcG).
\end{equation}
The results in \eqref{eq:inter1}--\eqref{eq:inter2} imply that $\mcN_w^{k-1}(\mcG)\neq \mcN_w^{k-1}(\mcG')$ for some leaf nodes $w$, which is a contradiction.
\end{IEEEproof}

\begin{IEEEproof}[Proof of Theorem~\ref{th:probe}]
If $\bb$ is feasible, it is also a minimizer of \eqref{eq:tv} if and only if it yields zero objective value as $\bb_o$ does. That happens when $\bR(\bb)\bI_{\mcC}\bDelta=\bR(\bb_o)\bI_{\mcC}\bDelta$. Since $\textrm{rank}(\bDelta)=C$, the latter implies $\bR(\bb)\bI_{\mcC}=\bR(\bb_o)\bI_{\mcC}$, or 
\begin{equation}\label{eq:Re=Re}
\bR(\bb) \be_m = \bR(\bb_o) \be_m, \forall m\in \mcC. 
\end{equation}
Heed that, if $\bR(\bb) \be_m=\bR(\bb_o) \be_m$, Lemma~\ref{le:entriesR2} (claim (ii)) implies that for $d_m(\bb) = d_m(\bb_o)$ and that, for all $k=1,\ldots,d_m$ 
\begin{equation*}
\mcN^{k}_m(\bb) = \mcN^{k}_m(\bb_o).
\end{equation*}
For every $n \in \mcN_m^k(\bb)$ and $s\in \mcN_m^{k-1}(\bb)$, \eqref{eq:Re=Re} implies
\begin{equation}\label{eq:same_diff}
[R(\bb)]_{mn} - [R(\bb)]_{ms} = [R(\bb_o)]_{mn} - [R(\bb_o)]_{ms}.
\end{equation}
Combining \eqref{eq:same_diff} with Lemma~\ref{le:entriesR2} (claim (iii)), provides
\begin{equation}
[R(\bb)]_{mn} - [R(\bb)]_{ms} = r_{\alpha_m^k(\bb_o) , \alpha_m^{k-1}(\bb_o)}.
\label{eq:same_parent}
\end{equation}
Under Assumption~\ref{ass:diff_resistances}, no two lines have the same resistance. Therefore, equation \eqref{eq:same_parent} is satisfied only if $\alpha_m^k(\bb) = \alpha_m^{k}(\bb_o)$ for $k=0,\ldots,d_m$, and hence, $\mcA_m(\bb) = \mcA_m(\bb_o)$ follows.

To prove sufficiency, suppose the conditions in \eqref{eq:probe} hold for a probing bus $m$. Equations \eqref{eq:probe:a} and \eqref{eq:entriesR1} ensure that $d_m(\bb) = d_m(\bb_o)$ and that  $[R(\bb)]_{mn} = [R(\bb_o)]_{mn}$ for all $n\in\mcA_m(\bb_o)$. On the other hand, Lemma~\ref{le:entriesR2} guarantees that $[R(\bb)]_{mn} = [R(\bb_o)]_{mn}$ for all $n\in\mcN^k_m({\bb})$ and $k=0,\ldots,d_m$.
It therefore follows that $\bR(\bb) \bI_{\mcC} = \bR(\bb_o) \bI_{\mcC}$.
\end{IEEEproof}

\bibliographystyle{IEEEtran}
\bibliography{main}

\begin{thebibliography}{10}
\providecommand{\url}[1]{#1}
\csname url@samestyle\endcsname
\providecommand{\newblock}{\relax}
\providecommand{\bibinfo}[2]{#2}
\providecommand{\BIBentrySTDinterwordspacing}{\spaceskip=0pt\relax}
\providecommand{\BIBentryALTinterwordstretchfactor}{4}
\providecommand{\BIBentryALTinterwordspacing}{\spaceskip=\fontdimen2\font plus
\BIBentryALTinterwordstretchfactor\fontdimen3\font minus
  \fontdimen4\font\relax}
\providecommand{\BIBforeignlanguage}[2]{{%
\expandafter\ifx\csname l@#1\endcsname\relax
\typeout{** WARNING: IEEEtran.bst: No hyphenation pattern has been}%
\typeout{** loaded for the language `#1'. Using the pattern for}%
\typeout{** the default language instead.}%
\else
\language=\csname l@#1\endcsname
\fi
#2}}
\providecommand{\BIBdecl}{\relax}
\BIBdecl

\bibitem{ExpConCanBook}
A.~G\'{o}mez-Exp\'{o}sito, A.~J. Conejo, and C.~Canizares, Eds., \emph{Electric
  Energy Systems, Analysis and Operation}.\hskip 1em plus 0.5em minus
  0.4em\relax Boca Raton, FL: CRC Press, 2009.

\bibitem{BoSch13}
S.~Bolognani, N.~Bof, D.~Michelotti, R.~Muraro, and L.~Schenato,
  ``Identification of power distribution network topology via voltage
  correlation analysis,'' in \emph{Proc. {IEEE} Conf. on Decision and Control},
  Florence, Italy, Dec. 2013, pp. 1659--1664.

\bibitem{Deka1}
D.~Deka, M.~Chertkov, and S.~Backhaus, ``Structure learning in power
  distribution networks,'' \emph{{IEEE} Trans. Control Netw. Syst.}, vol.~PP,
  no.~99, pp. 1--1, 2017.

\bibitem{Deka4}
D.~Deka, S.~Backhaus, and M.~Chertkov, ``Learning topology of distribution
  grids using only terminal node measurements,'' in \emph{Proc. {IEEE} Intl.
  Conf. on Smart Grid Commun.}, Syndey, Australia, Nov 2016.

\bibitem{ParkDeka}
P.~Sejun, D.~Deka, and M.~Chertkov, ``Exact topology and parameter estimation
  in distribution grids with minimal observability,'' in \emph{Proc. Power
  Syst. Comput. Conf.}, Dublin, Ireland, Jun. 2018.

\bibitem{WengLiaoRajagopal17}
Y.~Weng, Y.~Liao, and R.~Rajagopal, ``Distributed energy resources topology
  identification via graphical modeling,'' \emph{{IEEE} Trans. Power Syst.},
  vol.~32, no.~4, pp. 2682--2694, Jul. 2017.

\bibitem{7058419}
H.~Sedghi and E.~Jonckheere, ``Statistical structure learning to ensure data
  integrity in smart grid,'' \emph{{IEEE} Trans. Smart Grid}, vol.~6, no.~4,
  pp. 1924--1933, Jul. 2015.

\bibitem{7541005}
D.~Deka, S.~Backhaus, and M.~Chertkov, ``Estimating distribution grid
  topologies: A graphical learning based approach,'' in \emph{{Power Systems
  Computation Conf.}}, Genoa, Italy, Jun. 2016.

\bibitem{TalDeMate2017}
S.~Talukdar, D.~Deka, D.~Materassi, and M.~Salapaka, ``Exact topology
  reconstruction of radial dynamical systems with applications to distribution
  system of the power grid,'' in \emph{Proc. American Control Conf.}, Seattle,
  WA, May 2017.

\bibitem{CavArg2014}
G.~Cavraro and R.~Arghandeh, ``Power distribution network topology detection
  with time-series signature verification method,'' \emph{{IEEE} Trans. Power
  Syst.}, vol.~PP, no.~99, pp. 1--1, 2017.

\bibitem{sevlian2015distribution}
\BIBentryALTinterwordspacing
R.~Sevlian and R.~Rajagopal, ``Distribution system topology detection using
  consumer load and line flow measurements,'' Sep. 2017. [Online]. Available:
  \url{https://arxiv.org/pdf/1503.07224.pdf}
\BIBentrySTDinterwordspacing

\bibitem{CaKe2017}
G.~Cavraro, V.~Kekatos, and S.~Veeramachaneni, ``Voltage analytics for power
  distribution network topology verification,'' \emph{{IEEE} Trans. Smart
  Grid}, vol.~PP, no.~99, pp. 1--1, 2017.

\bibitem{zhao2017learning}
\BIBentryALTinterwordspacing
Y.~Zhao, J.~Chen, and H.~V. Poor, ``A learning-to-infer method for real-time
  power grid topology identification,'' Oct. 2017. [Online]. Available:
  \url{https://arxiv.org/pdf/1710.07818.pdf}
\BIBentrySTDinterwordspacing

\bibitem{patopa}
J.~Yu, Y.~Weng, and R.~Rajagopal, ``{PaToPa}: A data-driven parameter and
  topology joint estimation framework in distribution grids,'' \emph{{IEEE}
  Trans. Power Syst.}, vol.~PP, no.~99, 2017.

\bibitem{Ardakanian17}
\BIBentryALTinterwordspacing
Y.~Yuan, O.~Ardakanian, S.~Low, and C.~Tomlin, ``On the inverse power flow
  problem,'' Dec. 2017. [Online]. Available:
  \url{https://arxiv.org/abs/1610.06631}
\BIBentrySTDinterwordspacing

\bibitem{ErTpVi13}
T.~Erseghe, S.~Tomasin, and A.~Vigato, ``Topology estimation for smart micro
  grids via powerline communications,'' \emph{{IEEE} Trans. Signal Process.},
  vol.~61, no.~13, pp. 3368--3377, Jul. 2013.

\bibitem{Scaglione2017}
M.~Angjelichinoski, C.~Stefanovic, P.~Popovski, A.~Scaglione, and F.~Blaabjerg,
  ``Topology identification for multiple-bus {DC} microgrids via primary
  control perturbations,'' in \emph{{IEEE Intl. Conf. on DC Microgrids}},
  Nurnberg, Germany, Jun. 2017.

\bibitem{BheKeVe2017}
S.~Bhela, V.~Kekatos, and S.~Veeramachaneni, ``Enhancing observability in
  distribution grids using smart meter data,'' \emph{{IEEE} Trans. Smart Grid},
  vol.~PP, no.~99, 2017.

\bibitem{BKVZ17}
S.~Bhela, V.~Kekatos, L.~Zhang, and S.~Veeramachaneni, ``Enhancing
  observability in power distribution grids,'' in \emph{Proc. {IEEE} Intl.
  Conf. on Acoustics, Speech, and Signal Process.}, New Orleans, LA, Mar. 2017.

\bibitem{6203379}
M.~Nabi-Abdolyousefi and M.~Mesbahi, ``Network identification via node
  knockout,'' \emph{{IEEE} Trans. Autom. Contr.}, vol.~57, no.~12, pp.
  3214--3219, Dec. 2012.

\bibitem{6966724}
S.~Shahrampour and V.~M. Preciado, ``Topology identification of directed
  dynamical networks via power spectral analysis,'' \emph{{IEEE} Trans. Autom.
  Contr.}, vol.~60, no.~8, pp. 2260--2265, Aug. 2015.

\bibitem{GodsilRoyle}
C.~Godsil and G.~Royle, \emph{Algebraic Graph Theory}.\hskip 1em plus 0.5em
  minus 0.4em\relax New York, NY: Springer, 2001.

\bibitem{Saverio2}
S.~Bolognani and S.~Zampieri, ``On the existence and linear approximation of
  the power flow solution in power distribution networks,'' \emph{{IEEE} Trans.
  Power Syst.}, vol.~31, no.~1, pp. 163--172, Feb. 2015.

\bibitem{ansic84}
\emph{{C}84.1-1995 Electric Power Systems and Equipment Voltage Ratings (60
  {H}erz)}, ANSI Std., 2011.

\bibitem{pecandata}
\BIBentryALTinterwordspacing
(2013) {Pecan Street Inc. Dataport}. [Online]. Available:
  \url{https://dataport.pecanstreet.org/}
\BIBentrySTDinterwordspacing

\bibitem{chung1997spectral}
F.~Chung, \emph{Spectral graph theory}.\hskip 1em plus 0.5em minus 0.4em\relax
  Providence, RI: {American Mathematical Society}, 1997.

\bibitem{VKZG16}
V.~Kekatos, L.~Zhang, G.~B. Giannakis, and R.~Baldick, ``Voltage regulation
  algorithms for multiphase power distribution grids,'' \emph{{IEEE} Trans.
  Power Syst.}, vol.~31, no.~5, pp. 3913--3923, Sep. 2016.

\bibitem{kay93book}
S.~M. Kay, \emph{Fundamentals of Statistical Signal Processing, {V}ol. I:
  {E}stimation Theory}.\hskip 1em plus 0.5em minus 0.4em\relax Upper Saddle
  River, NJ: Prentice Hall, 1993.

\bibitem{GLasso}
J.~Friedman, T.~Hastie, and R.~Tibshirani, ``Sparse inverse covariance
  estimation with the graphical lasso,'' \emph{Biostatistics}, vol.~9, no.~3,
  pp. 432--441, Dec. 2008.

\bibitem{LiPoSc13}
X.~Li, V.~Poor, and A.~Scaglione, ``Blind topology identification for power
  systems,'' in \emph{Proc. {IEEE} Intl. Conf. on Smart Grid Commun.},
  Vancouver, BC, Canada, Oct. 2013.

\bibitem{KGB14}
V.~Kekatos, G.~B. Giannakis, and R.~Baldick, ``Grid topology identification
  using electricity prices,'' in \emph{Proc. {IEEE} Power \& Energy Society
  General Meeting}, Washington, DC, Jul. 2014.

\bibitem{egilmez2016graph}
H.~E. Egilmez, E.~Pavez, and A.~Ortega, ``Graph learning from data under
  {Laplacian} and structural constraints,'' \emph{{IEEE} J. Sel. Topics Signal
  Process.}, vol.~11, no.~6, pp. 825--841, Sep. 2017.

\bibitem{Ahuja}
R.~K. Ahuja, T.~L. Magnanti, and J.~B. Orlin, \emph{Network Flows: Theory,
  Algorithms, and Applications}.\hskip 1em plus 0.5em minus 0.4em\relax Upper
  Saddle River, NJ: Prentice Hall, 1993.

\bibitem{Be15}
D.~P. Bertsekas, \emph{Convex Optimization Algorithms}.\hskip 1em plus 0.5em
  minus 0.4em\relax Belmont, MA: Athena Scientific, 2015.

\bibitem{horn1990matrix}
R.~A. Horn and C.~R. Johnson, \emph{Matrix analysis}.\hskip 1em plus 0.5em
  minus 0.4em\relax Cambridge, UK: Cambridge University Press, 1990.

\bibitem{BW3}
M.~Baran and F.~Wu, ``Network reconfiguration in distribution systems for loss
  reduction and load balancing,'' \emph{{IEEE} Trans. Power Del.}, vol.~4,
  no.~2, pp. 1401--1407, Apr. 1989.

\bibitem{Kersting}
W.~H. Kersting, \emph{Distribution System Modeling and Analysis}.\hskip 1em
  plus 0.5em minus 0.4em\relax New York, NY: {CRC Press}, 2001.

\end{thebibliography}

\end{document}